\pdfoutput=1

\documentclass[11pt]{article}
\usepackage[utf8]{inputenc}  
\usepackage[british]{babel}
\usepackage{amsmath,amssymb,amsfonts,amsthm,bbm}
\usepackage{tikz}
\usepackage{enumerate,enumitem,mdframed}

\usepackage{crossreftools,booktabs}

\makeatletter
\newcommand{\optionaldesc}[2]{%
  \phantomsection
  #1\protected@edef\@currentlabel{#1}\label{#2}%
}
\makeatother

\usepackage{microtype}

\usepackage[mono=false]{libertine}
\usepackage[T1]{fontenc}
\usepackage[cmintegrals,libertine]{newtxmath}
\usepackage[cal=euler, calscaled=.95, frak=euler]{mathalfa}
\useosf
\usepackage[a4paper,vmargin={2cm,2cm},hmargin={2.25cm,2.25cm}]{geometry}
\usepackage[font=sf, labelfont={sf,bf}, margin=1cm]{caption}
\usepackage{numprint,ulem}
\newlist{compitem}{itemize}{4}
\setlist[compitem,1]{nolistsep,label=$\bullet$}

\usepackage{makecell} 

\SetLabelAlign{CenterWithParen}{\makebox[1.6cm]{#1}}

\usepackage{hyperref}
\hypersetup{
	colorlinks=true,
		citecolor=blue!60!black,
		linkcolor=red!60!black,
		urlcolor=green!40!black,
		filecolor=yellow!50!black,
	breaklinks=true,
	pdfpagemode=UseNone,
	bookmarksopen=false,
}

\colorlet{link}{red!60!black}

\theoremstyle{plain}
\newtheorem{theorem}{Theorem}[section]
\newtheorem{lemma}[theorem]{Lemma}
\newtheorem{proposition}[theorem]{Proposition}
\newtheorem{corollary}[theorem]{Corollary}

\theoremstyle{definition}
\newtheorem{definition}[theorem]{Definition}

\theoremstyle{remark}
\newtheorem{remark}[theorem]{Remark}

\makeatletter
\newcommand{\labeltext}[3][]{%
    \@bsphack%
    \csname phantomsection\endcsname
    \def\tst{#1}%
    \def\labelmarkup{\textcolor{link}}
    \def\refmarkup{}%
    \ifx\tst\empty\def\@currentlabel{\refmarkup{#2}}{\label{#3}}%
    \else\def\@currentlabel{\refmarkup{#1}}{\label{#3}}\fi%
    \@esphack%
    \labelmarkup{#2}
}
\makeatother
\newcommand{\N}{\mathbb{N}} 
\newcommand{\R}{\mathbb{R}} 
\renewcommand{\P}{{\mathbb{P}}} 
\newcommand{\E}{\mathbb{E}} 
\newcommand{\Lp}{\mathbb{L}} 
\newcommand{\1}{\mathbbm{1}} 

\newcommand{\esp}[1]{\mathbb{E}\left[#1\right]} 
\newcommand{\proba}[1]{\mathbb{P}\left(#1\right)} 

\newcommand{\e}{\varepsilon} 

\newcommand{\T}{\mathcal{T}} 
\newcommand{\F}{\mathcal{F}} 

\newcommand{\birth}{\mathsf{b}} 
\newcommand{\coal}{\mathsf{c}} 

\renewcommand{\H}{\mathsf{H}} 
\newcommand{\Height}{\mathsf{Height}} 
\newcommand{\h}{\mathsf{h}} 

\newcommand{\Xb}{\mathrm{\mathbf{x}}} 
\newcommand{\X}{\mathrm x} 
\newcommand{\XXb}{\mathrm{\mathbf{X}}} 
\newcommand{\XX}{\mathrm X} 

\newcommand{\V}{\mathbb{V}} 
\newcommand{\Fbb}{\mathbb{F}} 
\newcommand{\A}{\mathbb{A}} 

\newcommand{\dc}{D} 

\newcommand{\cv}[1][n]{\enskip\mathop{\longrightarrow}^{}_{#1 \to \infty}\enskip}
\newcommand{\cvloi}[1][n]{\enskip\mathop{\longrightarrow}^{(d)}_{#1 \to \infty}\enskip}
\newcommand{\cvps}[1][n]{\enskip\mathop{\longrightarrow}^{a.s.}_{#1 \to \infty}\enskip}
\newcommand{\cvproba}[1][n]{\enskip\mathop{\longrightarrow}^{\P}_{#1 \to \infty}\enskip}
\newcommand{\cvLp}[1][p]{\enskip\mathop{\longrightarrow}^{\Lp^{#1}}_{n \to \infty}\enskip}

\DeclareMathOperator*{\limit}{\longrightarrow}
\newcommand{\K}{\mathbb{K}}
\newcommand{\GH}{\text{GH}}
\DeclareMathOperator*{\supp}{supp}
\DeclareMathOperator*{\diam}{Diam}
\newcommand{\GP}{\text{GP}}

\newcommand{\GHP}{\text{GHP}}

\usepackage{mathtools}

\let\originalleft\left
\let\originalright\right
\renewcommand{\left}{\mathopen{}\mathclose\bgroup\originalleft}
\renewcommand{\right}{\aftergroup\egroup\originalright}

\DeclareSymbolFont{extraup}{U}{zavm}{m}{n}
\DeclareMathSymbol{\vardspade}{\mathalpha}{extraup}{81}
\DeclareMathSymbol{\varheart}{\mathalpha}{extraup}{86}
\DeclareMathSymbol{\vardiamond}{\mathalpha}{extraup}{87}
\DeclareMathSymbol{\varclub}{\mathalpha}{extraup}{84}

\makeatletter
\renewcommand*{\@fnsymbol}[1]{\ensuremath{\ifcase#1\or  \vardspade \or \varheart \or \vardiamond\or \varclub \or \bigstar \or
   \mathsection\or \mathparagraph\or \|\or **\or \dagger\dagger   \or \ddagger\ddagger \else\@ctrerr\fi}}
\makeatother

\author{
\'Etienne Bellin 
\thanks{CMAP, \'Ecole polytechnique, Institut Polytechnique de Paris, 91120 Palaiseau, France, \textsf{etienne.bellin@polytechnique.edu}
} \qquad  
Arthur Blanc-Renaudie 
\thanks{Tel Aviv University, Israel, \textsf{ablancrenaudiepro@gmail.com}. Supported by ERC consolidator grant 101001124 (UniversalMap). 
} 
\qquad  
Emmanuel Kammerer 
\thanks{CMAP, \'Ecole polytechnique, Institut Polytechnique de Paris, 91120 Palaiseau, France, \textsf{emmanuel.kammerer@polytechnique.edu}
}
 \qquad  
Igor Kortchemski 
\thanks{CNRS \& CMAP, \'Ecole polytechnique, Institut Polytechnique de Paris, 91120 Palaiseau, France, \textsf{igor.kortchemski@math.cnrs.fr}
} 
}

\title{Uniform attachment with freezing: Scaling limits}

\begin{document}
\date{}
\vspace{-2cm}
\maketitle 
\begin{abstract}
We investigate scaling limits of trees built by uniform attachment with freezing, which is a variant of  the classical model of random recursive trees introduced in a companion paper.  Here vertices are allowed to freeze, and arriving vertices cannot be attached to already frozen ones. We identify a phase transition when  
the number of non-frozen vertices roughly evolves  as the total number of vertices to a given power. In particular, we observe a critical regime where the scaling limit is a random compact real tree,  closely related to  a time non-homogeneous Kingman coalescent process identified by Aldous.  Interestingly, in this critical regime, a condensation phenomenon can occur.
\end{abstract}

\begin{figure}[!ht]
\label{fig:ssimus_intro}
\centering
\includegraphics[width=0.45\linewidth]{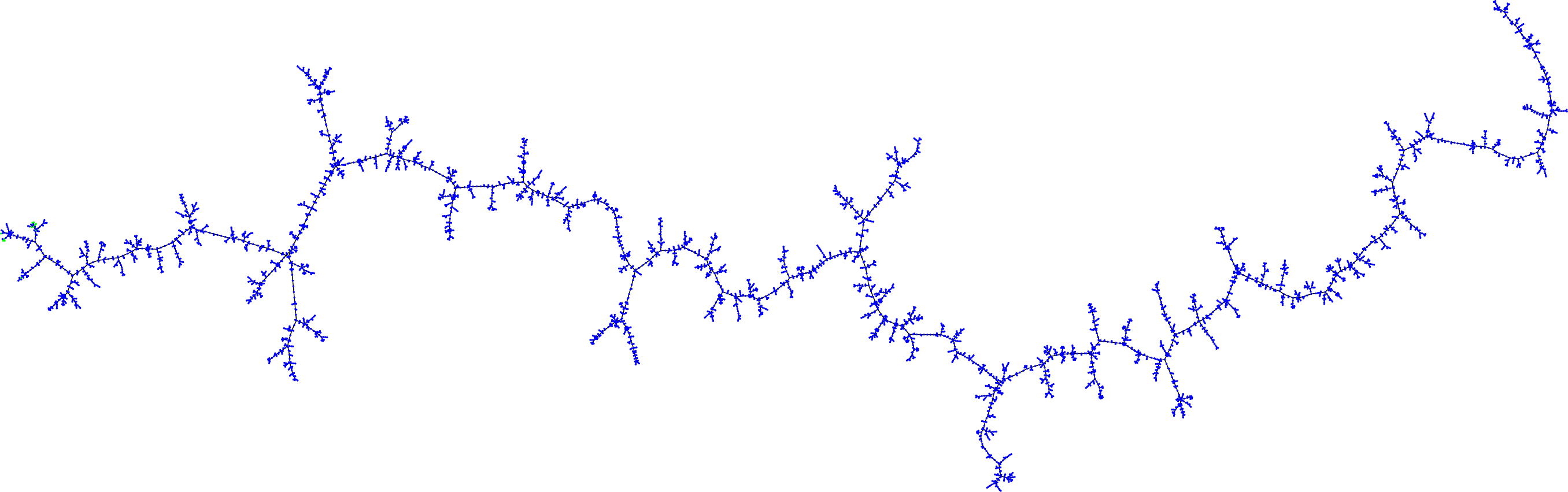}%
\qquad  \includegraphics[width=0.45\linewidth]{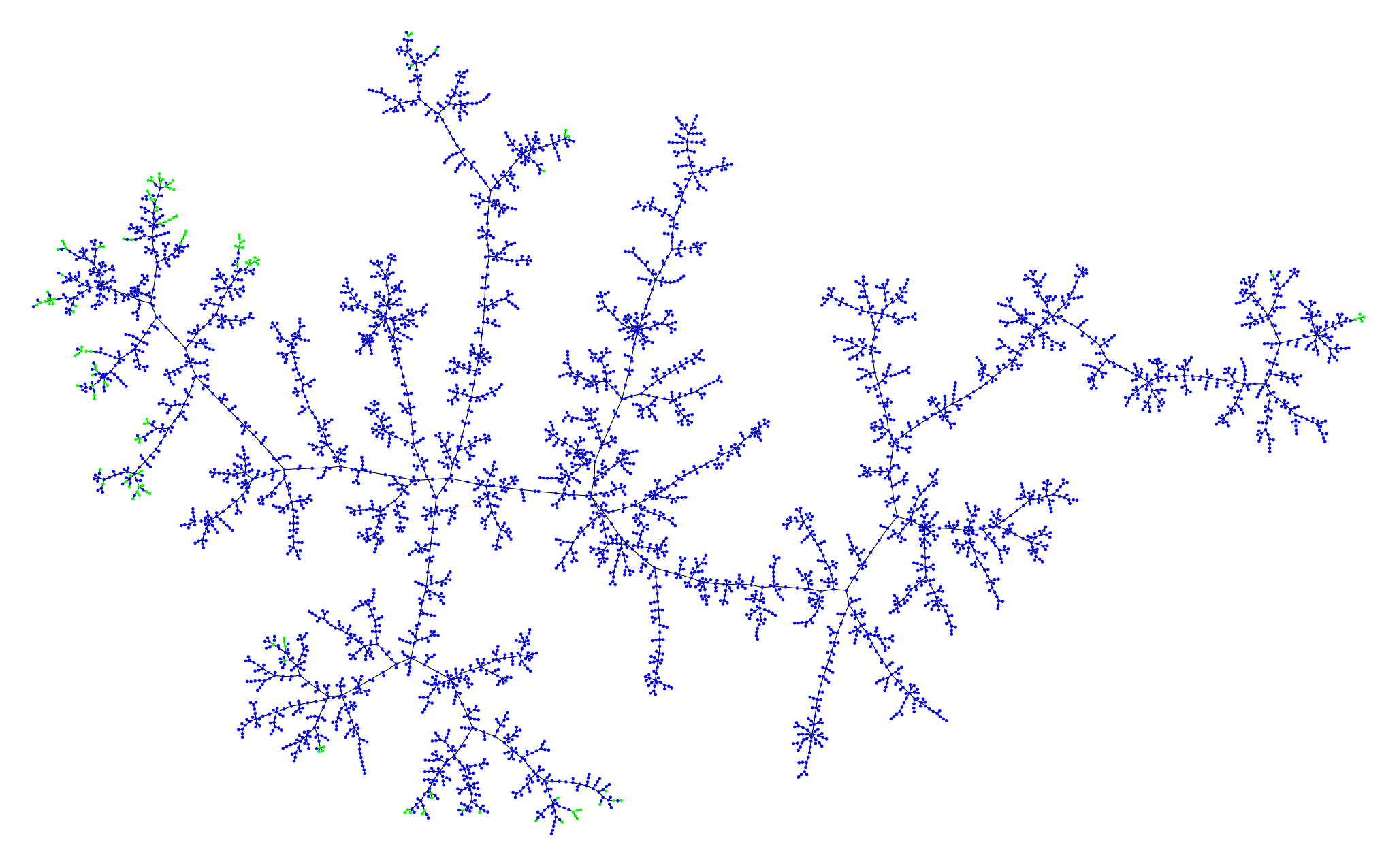} \\
\includegraphics[width=0.4\linewidth]{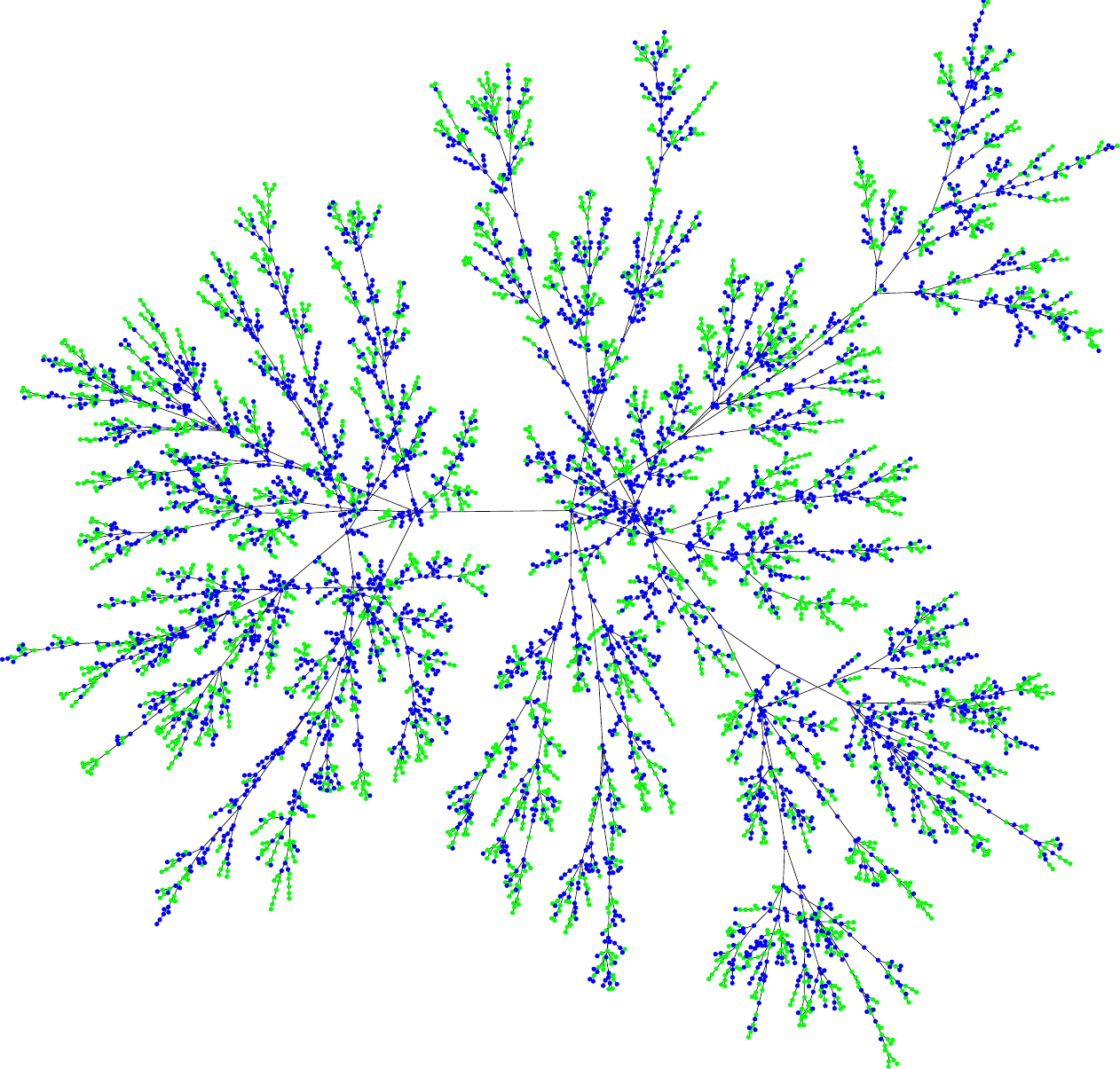}%
\qquad
\includegraphics[width=0.4\linewidth]{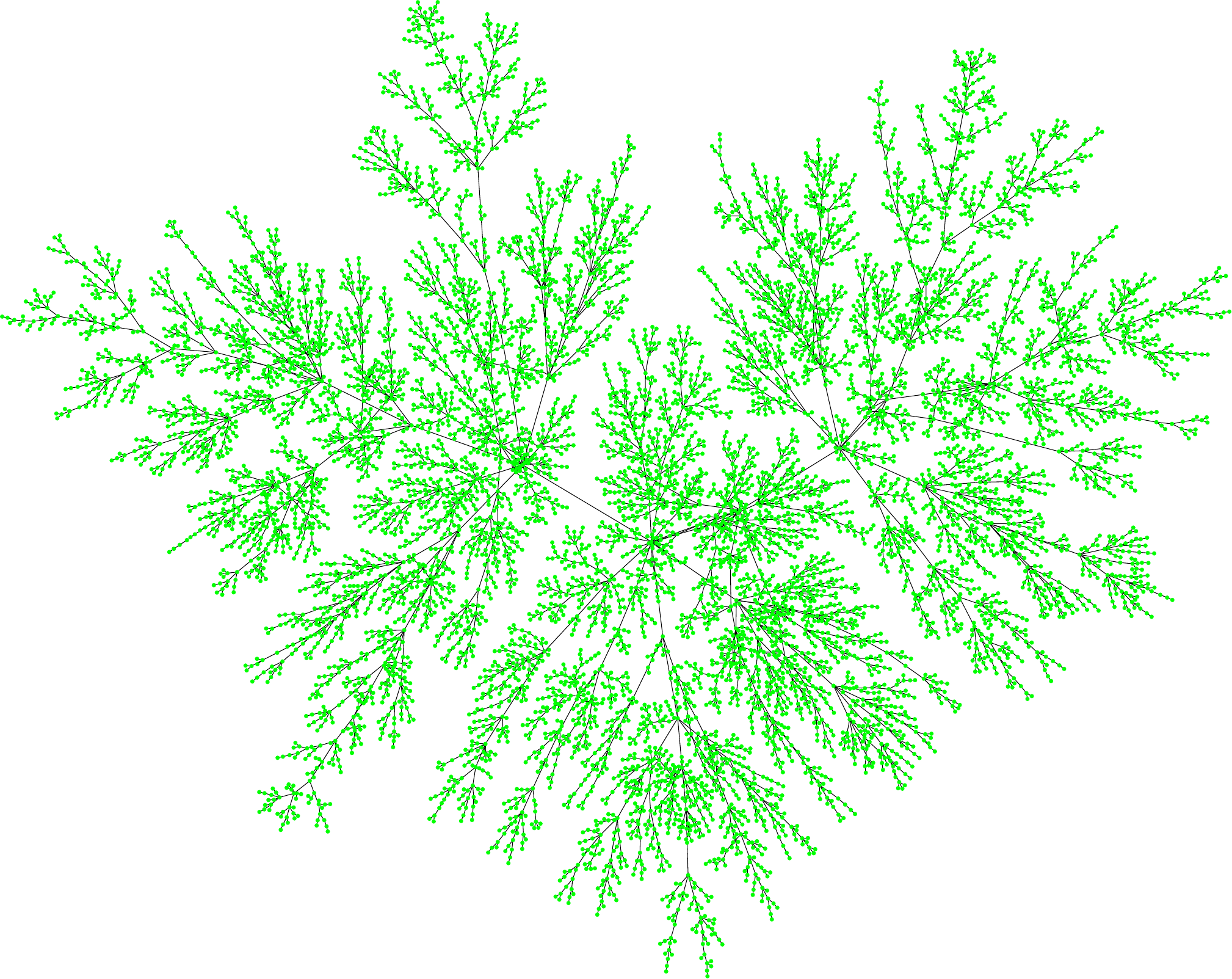}%
\caption{Simulations of the model of uniform attachment with freezing, when the number of active (i.e.~non-frozen) vertices roughly evolves  as the total number of vertices to the power $\alpha$; top: $\alpha=0.2,0.5$, bottom: $\alpha=0.8,1$. Frozen vertices are blue; active  vertices are  green.
}
\end{figure}

\clearpage
\tableofcontents

\section{Introduction}
\label{sec:intro}

We are interested in the geometry of uniform attachment trees with freezing, which we introduced in \cite{BBKK23+}. In the classical model of uniform recursive trees (sometimes also called  uniform attachment trees), trees are constructed recursively  starting with one single vertex, and successively attaching new vertices to a previous existing vertex, chosen uniformly at random. In the model of  uniform attachment trees with freezing, existing vertices can freeze and new vertices cannot be attached to frozen vertices.

The motivations in introducing this model were multiple. In the context of real-world networks such mechanisms naturally appear: for instance, on social media (such as Twitter) a user can choose to set their account to ``private'' which prevents strangers from ``following'' them; also performing an infection-tracing of an SIR epidemics falls within this framework (see \cite{BBKK23+}). Second, this model presents interesting mathematical features: for instance, it extends both uniform recursive trees and uniform plane trees, and in \cite{BBKK23+} universal bounds on the height have been obtained.

The main purpose of this paper is to study scaling limits of uniform attachment trees with freezing, in the specific regime where the number of ~non-frozen vertices roughly evolves  as the total number of vertices to a given power.

\paragraph{Uniform attachment with freezing.} Let us recall the model, which is parametrized by  a deterministic sequence $\Xb=(\X_i)_{i \geq 1}$  of elements of $\{-1,+1\}$. Starting from a unique active vertex, we recursively build random trees by reading the elements of the sequence one after the other, by applying a ``freezing'' step when reading $-1$ (which amounts to freezing an active vertex chosen uniformly at random) and a ``uniform attachment'' step when reading $+1$ (which amounts to attaching a new vertex to an active vertex chosen uniformly at random).  For every $n \geq 1$ we denote by $\T_n(\Xb)$ the random tree recursively built in this fashion after reading the first $n$ elements of $\Xb$ (see Sec.~\ref{ssec:def} for a precise definition and Fig.~\ref{fig exemple arbre recursif avec gel} for an example).

Let us comment on our choice of parametrization. It would have been possible to define the model starting with a random sequence $\Xb$, but the choice of a deterministic sequence defines a more general model.  {Our results can then be applied in the context of a random sequence.}

It is interesting to note that this model encompasses the two classical models of random recursive trees (when $\X_{i}=1$ for every $i \geq 1$) and random uniform plane trees (when the sequence $\XXb =(\XX_{i})_{i \geq 1}$ is a sequence of {non constant} i.i.d.~uniform random variables on $ \{-1,+1\}$, see \cite[{Theorem 2}]{BBKK23+}).

\paragraph{Scaling limits.} In order to make explicit the regime we are interested in, we need to introduce some notation. Given  a sequence $\Xb=(\X_i)_{i \geq 1}$  of elements of $\{-1,+1\}$, we set $S_{0}(\Xb) \coloneqq1$ and for every $n \geq 1$
\begin{equation}
\label{eq:defintro}S_n(\Xb) \coloneqq 1+ \sum_{i=1}^n \X_i; \qquad \tau(\Xb) \coloneqq \inf \{n \geq 1 : S_{n}(\Xb)=0\}.
\end{equation}
Observe that $S_{n}(\Xb)$ represents the number of active vertices of $\T_{n}(\Xb)$. In particular, $\tau(\Xb)$, if it is finite, is the first time when all the vertices are frozen, so that the tree does not evolve anymore. In order to consider large trees, we thus need this time to be large.  For this reason, we consider a sequence $(\Xb^n)_{n \geq 1}$ of sequences of elements of $\{-1,1\}$ such that $\tau(\Xb^n)>n$ for every $n \geq 1$ (if $\tau(\Xb)=\infty$ we can of course take $\Xb^{n}=\Xb$).

We investigate scaling limits for a particular class of sequences $(\Xb^{n})_{n \geq 1}$ where  the number of active vertices, that is $S_{n}(\Xb^{n})$, roughly evolves  as the total number of vertices to the power $\alpha$, for fixed $\alpha \in (0,1]$. {More precisely, we require $\displaystyle (n^{\alpha}/S_{nt}(\Xb^n))_{0 \leq t \leq 1}$ to converge in $\mathrm{L}^1([0,1])$ towards some function $1/f$ together with an $\mathrm{L}^2$-type condition} (see Sec.~\ref{sec:results} for precise conditions). It turns out that the geometry of $\T_{n}(\Xb^{n})$ can be understood within this class, through scaling limits for which we identify different regimes according to the value of $\alpha$. In order to keep this introduction at reasonable length, we postpone the precise statements to Sec.~\ref{sec:results}, and summarize here the essence of our results. 
 \begin{enumerate}[noitemsep,nolistsep]
 \item[--] \emph{Subcritical regime $\alpha \in (0,1/2)$}:  $ \T_{n}(\Xb^{n})/n^{1-\alpha}$ converges in distribution to a deterministic line-segment for the Gromov--Hausdorff--Prokhorov topology (Theorem \ref{thm:sous-critique});
 \item[--] \emph{critical regime $\alpha=1/2$}: $\T_{n}(\Xb^{n})/ \sqrt{n}$ converges in distribution to a random compact measured real tree for the Gromov--Hausdorff--Prokhorov topology, with a condensation phenomenon when $1/{f^{2}}$ is integrable near $0$ (Theorem \ref{thm:critique});
 \item[--] \emph{supercritical regime $\alpha \in (1/2,1)$}:  $ \T_{n}(\Xb^{n})$ looks like a ``star''  with  branches of random lengths of order $n^{1-\alpha}$ (Theorem \ref{thm:sur-critique}).
 \end{enumerate}
 Here the terminology ``$X$ looks like $Y$'' means that the joint laws of the graph-distances of $k$ points chosen in an i.i.d. way from $X$ and from $Y$ are ``close''.

\paragraph{{Motivation.}} {Let us give some further motivation:
\begin{enumerate}[noitemsep,nolistsep]
\item[--] it is natural to dvelve into the mathematical aspects of freezing's impacts in dynamically-built random graph models. Our goal is to explore whether a regime exists where scaling limits can be established for the GHP topology (for recursive trees, without freezing, there are notably no scaling limits for the GHP topology);
\item[--] The choice to focus on the quantity $S_{n}$ is natural, since it represents the number of active vertices at time $n$. Also, in the context of the study of the so-called ``infection tree'' of a stochastic SIR dynamics, in which the vertices are individuals and where edges connect two individuals if one has infected the other, the quantity $S_{n}$ represents the number of infected vertices at time $n$;
\item[--] our assumption  that $S_{n}(\Xb^{n})$ roughly evolves  as the total number of vertices to the power $\alpha$ is motivated by the following three compelling reasons: it includes the case where $(S_k(\Xb^{n}))_{0 \le k \le 2n+1}$ is an excursion of a simple random walk of length $2n+1$ (in this case $\alpha=1/2$), when $\alpha=1/2$ the scaling limit closely related to a time non-homogeneous Kingman coalescent process identified by Aldous, and finally it involves a one-parameter family exhibiting  a phase transition;
\item[--] one of the motivations to consider a ``triangular array'' setting, where we consider $S_{n}(\Xb^{n})$ instead of $S_{n}(\Xb)$, due to the need to accommodate cases involving excursions of a simple random walk with lengths dependent on $n$ {(see Sec.~\ref{ssec:CRT})};
\item[--] the assumption that $\displaystyle (n^{\alpha}/S_{nt}(\Xb^n))_{0 \leq t \leq 1}$ converges, in a certain sense, towards some function $1/f$ is motivated by its applicability in cases where  $\Xb^{n}$ is an excursion of a a simple random walk (in this case $\alpha=1/2$ and $f$ is a normalized Brownian excursion). Addititionally,  it exhibits a universal behavior.
\end{enumerate}
}

\paragraph{{Main ideas.}}
Our proof is based on an alternative construction of uniform attachment trees with freezing, which we introduced in \cite{BBKK23+}: by reversing time,  uniform attachment trees with freezing can be built using a growth-coalescent process of rooted forests. This extends the well-known connection between recursive trees and the Kingman coalescent {which first appeared in \cite{DR76} (see \cite[Sec.~6]{Dev87}, \cite[Sec.~3]{Pit94}, \cite[Sec.~2.2]{AB15}, \cite{ABE18,Esl21,Esl22} for applications)}. This construction enables us to control distances using classical concentration tools combined with the chaining method. This method is an important technique in concentration theory \cite{Tal05} whose goal is to estimate the maximum of a given function on a given space using a sequence of increasing subspaces. It has found many applications in the study of random metric spaces, see e.g. \cite{Ald91,ADGO17,CH17,Sen19,Bla22,Bla22+}. Moreover, in the case $\alpha=1/2$, the continuous analogue of the alternative construction can be viewed as a time-change of a time non-homogeneous Kingman coalescent process introduced by Aldous.

\paragraph{{{Relation with \cite{BBKK23+}.}}} {Let us make as explicit as possible the dependence of this work on the companion paper \cite{BBKK23+}. Our main results here, concerning scaling limits, only rely on the alternative time-reversed construction of uniform attachment trees with freezing introduced in \cite{BBKK23+} and some easy consequences that follow (recalled in Theorem \ref{thm:samelaw} and Lemmas \ref{lem:accroissementH}, \ref{lem:bunif} and \ref{loi du temps de coalescence} below, as well as \cite[Theorem 3 (3)]{BBKK23+}). This put aside, our two papers are disjoint: here we are interested in the case $\alpha<1$, while the companion paper \cite{BBKK23+} explores two different directions: first, it obtains universal bounds on the height of the tree, and second the study of what can be seen as the boundary case $\alpha=1$ (in this case, $\T_{n}(\Xb^{n})$ looks like a ``{tentacular} bush'' in the sense that {at the first order }two typical vertices are always at the same {deterministic} distance of order
 $\ln(n)$, {smaller that the height of the tree}).
} 

\paragraph{Plan of the paper.} We start in Sec.\ref{sec:results} by stating our main results concerning scaling limits in various regimes.  Sec.~\ref{sec:discret} then recalls 
an alternative construction of $\T_{n}$, based on time-reversal, through a growth-coalescent process of rooted forests, which we have introduced in \cite{BBKK23+}. We then obtain in Sec.~\ref{sec:distances} general estimates for distances in $\T_{n}$, and introduce in particular another distance based on coalescence times. This enables us to establish our scaling limit results, first in the non-critical cases $\alpha \neq 1/2$ in Section \ref{sec:noncritical} and then in the critical case $\alpha=1/2$ in Section \ref{sec:critical} {(in particular Sec.~\ref{sec:deflim} defines the limiting object and gives a more detailed description of the proof approach)}. Finally, Appendix \ref{sec:background} contains background on different topologies that we use and Appendix \ref{sec:leaf} contains a tightness criterion.

\paragraph{Acknowledgments.} We are grateful to Christina Goldschmidt for stimulating discussions at early stages of this work. We thank the wonderful organization committee (Serte Donderwinkel, Christina Goldschmidt, Remco van der Hofstad, and Joost Jorritsma) of the RandNET Summer School and Workshop on Random Graphs, where this work was initiated.  {We also thank the referee for their careful reading and useful comments.}

\section{Scaling limits: main results}
\label{sec:results}

We consider a sequence $(\Xb^n)_{n \geq 1}$ of sequences of elements of $\{-1,1\}$ such that $\tau(\Xb^n)>n$ for every $n \geq 1$ (if $\tau(\Xb)=\infty$ we can of course take $\Xb^{n}=\Xb$ for every $n \geq 1$), and set $\T_{n}=\T_{n}(\Xb^{n})$. We keep the notation introduced in Sec.~\ref{sec:intro}, recall in particular \eqref{eq:defintro}. To simplify the notation, we set $S_k^n = S_k(\Xb^n)$ for every  $0 \leq  k\leq n$ and set $S_{t}^{n} = S_{\lfloor t \rfloor}^n $ for all real number $t \in [0,n+1)$.
We also set
\[ \h_n^+=\sum_{i=1}^n \frac{1}{S^{n}_i} \mathbbm 1_{\{\X^{n}_i=1\}}.\]
{The reason why the quantity $ \frac{1}{S^{n}_i} \mathbbm 1_{\{\X^{n}_i=1\}}$ appears is essentially that it represents the probability that at step $i$ the degree of a given active vertex changes.}

We study the geometry of $\T_{n}$ under certain assumptions, which we now state. Given $\alpha\in (0,1)$ and a measurable function $f: [0,1] \rightarrow \R_{+}$ such that $1/f \in \mathrm{L}^1([0,1]) \cap \mathrm{L}^2_\mathrm{loc}((0,1))$, we consider the following set of conditions:

\vbox{
	\begin{mdframed}[linecolor=black!60]
		\begin{description}[topsep=0pt,itemsep=-1ex,partopsep=1ex,parsep=1ex,labelwidth=1.6cm,leftmargin=!,align=CenterWithParen]
			\item[\textcolor{link}{\optionaldesc{$(\mathsf{Lim}^{\alpha}_{1/f})$}{hyp:conv1surf}}] $\displaystyle (n^{\alpha}/S^n_{ nt })_{0 \leq t \leq 1}$  converges in $\mathrm{L}^1([0,1])$ towards $1/f$;
			\item[\textcolor{link}{\optionaldesc{$(\mathsf{Lim}^{\alpha}_{1/f^2})$}{hyp:conv1surf2}}]  For all $0<a<b<1$ we have $\displaystyle\lim_{n\to \infty} \int_a^b \left(\frac{n^\alpha}{S^n_{nt}} \right)^2 dt
			= \int_a^b \frac{\mathrm{d}t}{f(t)^2}.$
		\end{description}
	\end{mdframed}
}
Notice that given \ref{hyp:conv1surf}, the condition \ref{hyp:conv1surf2} is equivalent to the convergence of $(({n^\alpha}/{S^n_{nt}} )^2)_{t \in [0,1]}$ towards $1/f^2$ in $\mathrm{L}^1_{\mathrm{loc}}((0,1))$. {Roughly speaking, the condition \ref{hyp:conv1surf} is related to the convergence of the heights of vertices, while the condition \ref{hyp:conv1surf2} is related to the behavior of distances between pairs of vertices.  Besides, one can show that when $\alpha \ge 1/2$, the condition \ref{hyp:conv1surf2} implies that for all $0<a<b<1$, the minimum $\min_{an\le k \le bn} S^n_k$ tends to $\infty$ as $n \to \infty$. It is for instance a consequence of Equation \eqref{eq:cvminSinfini} in the proof of Lemma \ref{lem:hnalpha}.}

We first identify the height of $\T_{n}$ under these assumptions:

\begin{theorem}
\label{thm:heightalpha}
Assume that $\alpha\in{(0,1)}$ and that $f:[0,1]\mapsto \R^+$ satisfies
\ref{hyp:conv1surf} and \ref{hyp:conv1surf2}. Then for all $p\geq1$:
\[
\frac{\Height(\T_{n})}{n^{1-\alpha}} \cvLp[p]  \int_{0}^{1} \frac{\mathrm{d}t}{2f(t)}.
\]
\end{theorem}
This follows from \cite[Theorem 3 (3)]{BBKK23+}, in virtue of which the convergence $\Height(\T_{n})/\h_{n}^{+} \rightarrow 1$ holds in $\mathbb{L}^{p}$, combined with estimating  the asymptotic behavior of $\h_{n}^{+}$.

\subsection{Subcritical regime \texorpdfstring{$\alpha \in (0,1/2)$}{}}

\begin{theorem}
\label{thm:sous-critique}
Assume that $\alpha\in{(0,1/2)}$ and that $f:[0,1]\mapsto \R^+$ satisfies \ref{hyp:conv1surf} and \ref{hyp:conv1surf2}. Then the convergence
\[
\frac{1}{n^{1-\alpha}} \cdot \T_{n}  \quad \mathop{\longrightarrow}^{(d)}_{n \rightarrow \infty} \quad \left[ 0, \int_0^1\frac{\mathrm{d}t}{2f(t)} \right]
\]
holds in distribution for the Gromov--Hausdorff--Prokhorov topology, where the segment $\left[ 0, \int_0^1 1/(2f(t)) {\mathrm{d}t}  \right]$ is equipped with the image measure of the uniform measure on $[0,1]$ by the mapping $x\mapsto \int_0^x\frac{\mathrm{d}t}{2f(t)}$.
\end{theorem}

We refer to Sec.~\ref{sec:background} for details concerning the  Gromov--Hausdorff--Prokhorov topology. See Theorem \ref{thm:alpha0} for the case  where $\| S^{n} \|_{\infty}$ is bounded (which roughly speaking may be viewed as the case $\alpha=0$).

  \subsection{Critical regime \texorpdfstring{$\alpha=1/2$}{}}

\begin{theorem}
\label{thm:critique}
Assume that $\alpha=1/2$, that $f:[0,1]\mapsto \R^+$ satisfies  \ref{hyp:conv1surf} and \ref{hyp:conv1surf2} and that $\max_{1 \leq  k \leq n} S^{n}_{k}= \mathcal{O}(\sqrt{n})$. Then the convergence
\[\frac{1}{\sqrt{n}} \cdot \T_{n}  \quad \mathop{\longrightarrow}^{(d)}_{n \rightarrow \infty} \quad \mathcal{T}(f)\]
holds in distribution for the Gromov--Hausdorff--Prokhorov topology, where $\mathcal{T}(f)$ is a random measured compact real tree.
\end{theorem}

The proof is divided in two steps: first the identification of the limit through finite-dimensional distributions by constructing a continuous analog of Algorithm \ref{algo2}, second establishing tightness for the Gromov--Hausdorff-Prokhorov topology by relying on the  chaining method.  {Let us mention that the convergence of finite-dimensional distributions holds without the assumption $\max_{1 \leq  k \leq n} S^{n}_{k}= \mathcal{O}(\sqrt{n})$, which is only used in the proof of tightness. It is unclear whether this assumption can be removed.}

We show that this tree is similar in many aspects to the Brownian tree introduced by Aldous \cite{Ald91} (see Sec.~\ref{sec:critical} for background concerning real trees). Specifically, we establish the following properties concerning the geometry of $ \mathcal{T}(f)$:

\pagebreak[2]
\begin{theorem}
\label{thm:properties}
The following assertions hold almost surely.
\begin{enumerate}
\item[(1)] The mass measure on $\T(f)$ has full support and gives full measure to the set of its leaves.
\item[(2)] $\T(f)$ is binary if and only if $\int_0^{1/2} \frac{1}{f^2}=\infty$. When $\int_0^{1/2} \frac{1}{f^2}<\infty$, the root has infinite degree and all the other branchpoints are binary.
\end{enumerate}
\end{theorem}

Observe that, quite surprisingly, {a condensation phenomenon may occur; in particular} the {topological} class of $\T(f)$ depends on the integrability of $1/f^{2}$ near $0$. {This behavior is new among random trees with ``large'' degrees, such as stable trees \cite{LGLJ98,DLG02} (where there are infinitely many vertices with ``large'' degrees, with existence of GHP scaling limits), non-generic Bienaymé trees \cite{JS11,Jan12,Kor15}  and Cauchy-Bienaymé-trees \cite{KR19} (where there is a unique vertex with macroscopic degree, but without existence of nontrivial GHP scaling limits), simply generated trees with superexponential branching weights \cite{JJS11} (where there is a unique vertex with macroscopic degree, but without existence of nontrivial  GHP scaling limits).}

Let us mention that it is not too hard to check that  almost surely the Minkowski, packing and Hausdorff  dimensions of $\mathcal{T}(f)$ are at least equal to $2$. However, the upper dimensions depend on the the fine behavior of $f$ near $0$, and will be investigated in future work.

It turns out that by appropriately choosing $f$ one can recover the Brownian CRT.
\begin{corollary}
\label{cor:CRT}
If $\mathbbm{e}$ is a normalized Brownian excursion then $\T(\mathbbm{e})$ has the same law as the Brownian CRT.
\end{corollary}
{Let us comment on the fact that by $\T(\mathbbm{e})$ we mean a random variable whose law is characterized in the following ``annealed'' way:
\[
\esp{F(\T(\mathbbm{e}))}= \int \esp{F (\T(f))} \P_{\mathbbm{e}}(\mathrm{df}),
\]
for every nonnegative functional $F$, where $\P_{\mathbbm{e}}$ is the law of $\mathbbm{e}$.
}
Corollary \ref{cor:CRT} comes from Theorem \ref{thm:critique} applied with a carefully chosen random sequence $(\Xb^n)_{n \geq 1}$ (see Sec.~\ref{ssec:CRT} for details). {This result can also be seen as a consequence of {\cite[Theorem 2]{A98}}.}

\subsection{Supercritical regime \texorpdfstring{$\alpha \in (1/2,1)$}{}}

We denote by $d^{n}$ the graph distance on the vertices of $\T_{n}$.

\begin{theorem}
\label{thm:sur-critique}
Assume $\alpha \in (1/2,1)$ and $f:[0,1]\mapsto \R^+$ satisfies \ref{hyp:conv1surf} and \ref{hyp:conv1surf2}. 
For every $k \geq 2$, conditionally given  $\T_{n}$, let $V^n_1,V^n_2,\ldots, V^n_k$ be independent uniform vertices in $\T_{n}$ {and let $U_1, \ldots, U_n$ be i.i.d. uniform random variables in $[0,1]$}. Also let $V_0^n$ be the root of  $\, \T_n$ and let $U_0=0$. Then the following convergence holds in distribution.
\[
\left( \frac{1}{n^{1-\alpha}} \cdot d^n(V^n_i,V^n_j) \right)_{0\leq i,j\leq k}\cvloi \left(  \int_{0}^{U_{i}}\frac{\mathrm{d}t}{2f(t)} +  \int_{0}^{U_{j}}\frac{\mathrm{d}t}{2f(t)}  \right)_{0\leq i,j\leq k}
\]
{where $d^n$ denotes the graph distance in the tree $\T_n$.}
\end{theorem}
Intuitively speaking, this tells us that the tree $\T_{n}$ roughly looks like a ``star'' with branches of length $n^{1-\alpha} \int_0^1 1/(2f(t)) {\mathrm{d}t}$.  It is worth noting that  this case gives a natural example of trees growing polynomially in their size, without scaling limits. Indeed, $(\T_{n}/n^{1-\alpha})_{n \geq 1}$ is not tight for the GHP topology (since for $\varepsilon>0$ the number of balls of radius $\varepsilon$ needed to cover this space tends to infinity in probability). However, for another topology closely related to the Gromov--Prokhorov topology (see \cite{ET21,Jan21}),  $\T_{n}/n^{1-\alpha}$ converges to the so-called long dendron $\Upsilon_{\nu}$ (we use the notation of \cite[Example 3.12]{Jan21}) with $\nu$ being the image measure of the uniform measure on $[0,1]$ by the mapping $x \mapsto \int_0^x 1/(2f(t)) {\mathrm{d}t} $.

\section{Trees constructed by uniform attachment with freezing}
\label{sec:discret}

Here we give a precise definition of trees constructed by uniform attachment with freezing, and recall from \cite{BBKK23+} an alternative growth-coalescence algorithm to generate them. We also provide Table \ref{tab:secdiscret} which contains the notation that will be used in the future sections.

\begin{table}[htbp]\caption{Table of the main notation and symbols introduced in Section \ref{sec:discret} and used later.}
\centering
\begin{tabular}{c c p{12cm} }
\toprule
$\N$ && $=\{1,2,3,\dots\}$ positive integers\\
$[n]$ && $=\{1,2,\dots,n\}$ integers between $1$ and $n$\\
$\#A$ && cardinality of a finite set $A$\\
\hline
$\Xb = (\X_n)_{n\in\N}$ && a sequence of elements of $\{-1,1\}$\\
$S_n(\Xb)$ && $=1+\sum_{i=1}^n \Xb_i$\\
$\tau(\Xb)$ && $=\inf \{n \geq 1 : S_{n}(\Xb)=0\}$ \\
\hline
$\T_n(\Xb)$ && tree built at time $n$ by Algorithm \ref{algo1}; $S_{n}(\Xb)$ is its number of active vertices \\
$N_n(\Xb)$ && total number of vertices in $\T_n(\Xb)$; $N_n(\Xb)=(S_n(\Xb)+n+1)/2$ when $n \leq \tau(\Xb)$\\
\hline
$\F_n^n(\Xb),\F_{n-1}^n(\Xb), \ldots, \F_0^n(\Xb)$ && the forest of trees built by Algorithm \ref{algo2}\\

$\T^n(\Xb)$ && $=\F_0^n(\Xb)$ the output of Algorithm \ref{algo2}\\
\hline
$\Fbb_{n}(\Xb)$ && $=\{i\in\llbracket 1,n \rrbracket: \X_{i}=-1\}$ the labels of frozen vertices of $\T^n(\Xb)$\\

$\A_{n}(\Xb)$ && $=\{a_{1},\ldots,a_{S_{n}(\Xb)}\}$ the labels of active vertices of $\T^n(\Xb)$\\

$\V_{n}(\Xb)$ && $=\Fbb_{n}(\Xb)\cup\A_{n}(\Xb)$ the labels of all vertices of $\T^n(\Xb)$\\
\hline
$\birth_n(u)$ && the birth time of $u \in \V_{n}(\Xb)$ in the construction of $\T^n(\Xb)$  by Algorithm \ref{algo2} \\

$\coal_{n}(u,v)$ && the coalescence time between $u,v \in \V_{n}(\Xb)$ in the construction of $\T^n(\Xb)$  by Algorithm \ref{algo2} \\
\hline
 $\H^n_i(u)$  && the height of vertex $u$ in $\F^n_i(\Xb)$\\
\bottomrule
\end{tabular}
\label{tab:secdiscret}
\end{table}

\subsection{Uniform attachment with freezing: recursive construction}
\label{ssec:def}

Given a sequence $\Xb=(\X_n)_{n \geq 1}$  of elements of $\{-1,1\}$, for every $n \geq 1$ the tree $\T_n(\Xb)$ is built by  reading the first $n$ elements of the sequence $\Xb$, namely $\X_1,\dots,\X_n$. Here the trees will be rooted, vertex-labelled and edge-labelled; edges have the label corresponding to their time of appearance and vertices have the label corresponding to their time of freezing or the label ``$a$'' if they are still active (meaning they have not frozen yet). We set $S_{0}(\Xb) \coloneqq1$ and for every $n \geq 1$
\begin{equation}
\label{eq:defSn}S_n(\Xb) \coloneqq 1+ \sum_{i=1}^n \X_i; \qquad \tau(\Xb) \coloneqq \inf \{n \geq 1 : S_{n}(\Xb)=0\}.
\end{equation}

\paragraph{\labeltext[1]{Algorithm 1.}{algo1}}
\begin{compitem}
\item Start with the tree $\T_0(\Xb)$ made of a single root vertex labelled $a$ (vertices labelled $a$ are called active vertices, and the others frozen vertices).
\item For every $n \geq 1$, if $\T_{n-1}(\Xb)$ has no vertices labelled $a$, then set $\T_n(\Xb) = \T_{n-1}(\Xb)$. Otherwise let $V_n$ be a random uniform active vertex of $\T_{n-1}(\Xb)$, chosen independently from the previous ones. Then:
\begin{compitem}
\item[--] if $\X_n=-1$, build $\T_n(\Xb)$ from $\T_{n-1}(\Xb)$ simply by replacing the label $a$ of $V_n$ with the label $n$;
\item[--] if $\X_n=1$;  build $\T_n(\Xb)$ from $\T_{n-1}(\Xb)$ by adding an edge labelled $n$ between $V_n$ and a new vertex labelled $a$. 
\end{compitem}
\end{compitem}
{For $n \geq 0$, we view $\T_n(\Xb)$ as a rooted, double-labelled tree (that is edge-labelled and vertex-labelled). If $N_{n}(\Xb)$ represents the total number of vertices  of $\T_n(\Xb)$, observe that by construction $N_n(\Xb)=(S_n(\Xb)+n+1)/2$ for $0 \leq n \leq \tau(\Xb)$.} 

The law of the random tree $\T_n(\Xb)$ obviously depends on the sequence $\Xb$ (actually it only depends on the first $n$ elements $\X_1,\dots,\X_n$). Notice that, by construction, the sequence of trees $(\T_n(\Xb))_{n\geq 0}$ is  non-decreasing and reaches a stationary state if and only if $\tau(\Xb)<\infty$, in which case it reaches a stationary state for the first time at $\tau(\Xb)$. In this case, $\T_n(\Xb)$ has no more active vertices for the first time at $n = \tau(\Xb)$.

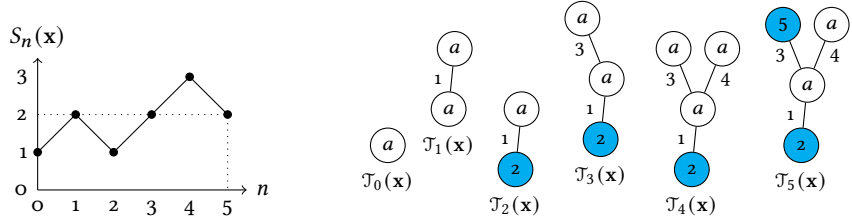
\begin{figure}
\begin{center}
\begin{tikzpicture}[scale = 0.5]
\draw[->] (0,0) -- (0,3.5);
\draw[->] (0,0) -- (5.5,0);

\node[above,font=\footnotesize] at (0,3.5) {$S_n(\Xb)$};
\node[right,font=\footnotesize] at (5.5,0) {$n$};

\node[draw,circle,fill,inner sep = 1pt] at (0,1) {};
\node[draw,circle,fill,inner sep = 1pt] at (1,2) {};
\node[draw,circle,fill,inner sep = 1pt] at (2,1) {};
\node[draw,circle,fill,inner sep = 1pt] at (3,2) {};
\node[draw,circle,fill,inner sep = 1pt] at (4,3) {};
\node[draw,circle,fill,inner sep = 1pt] at (5,2) {};

\draw (0,1)--(1,2)--(2,1)--(3,2)--(4,3)--(5,2);

\node[left,font=\footnotesize] at (0,0) {0};
\node[left,font=\footnotesize] at (0,1) {1};
\node[left,font=\footnotesize] at (0,2) {2};
\node[left,font=\footnotesize] at (0,3) {3};

\node[below,font=\footnotesize] at (0,0) {0};
\node[below,font=\footnotesize] at (1,0) {1};
\node[below,font=\footnotesize] at (2,0) {2};
\node[below,font=\footnotesize] at (3,0) {3};
\node[below,font=\footnotesize] at (4,0) {4};
\node[below,font=\footnotesize] at (5,0) {5};

\draw[dotted] (0,2)--(5,2)--(5,0);
\end{tikzpicture}
\hspace{2em}
\begin{tikzpicture}[scale = 0.8,
sommet/.style = {draw,circle, font=\scriptsize,inner sep=0,minimum size=13pt},
gele/.style = {fill=cyan},
etiquete/.style = {font = \scriptsize}]

\node[sommet] (0) at (-0.5,0.4) {$a$};
\node[etiquete] (t0) at (-0.5,-0.2) {$\T_0(\Xb)$};

\node[sommet] (1) at (0.5,1) {$a$};
\node[sommet] (2) at (0.6,2) {$a$};
\draw (1)--node[left,etiquete] {1} (2);
\node[etiquete] (t1) at (0.5,0.4) {$\T_1(\Xb)$};

\node[sommet,gele] (3) at (1.6,0) {2};
\node[sommet] (4) at (1.7,1) {$a$};
\draw (3) --node[left,etiquete] {1} (4);
\node[etiquete] (t2) at (1.6,-0.6) {$\T_2(\Xb)$};

\node[sommet,gele] (5) at (3,0.5) {2};
\node[sommet] (6) at (3.1,1.5) {$a$};
\node[sommet] (7) at (2.7,2.5) {$a$};
\draw (5) --node[left,etiquete] {1} (6);
\draw (6) --node[left,etiquete] {3} (7);
\node[etiquete] (t3) at (3,-0.1) {$\T_3(\Xb)$};

\node[sommet,gele] (8) at (4.5,0) {2};
\node[sommet] (9) at (4.6,1) {$a$};
\node[sommet] (10) at (4.2,2) {$a$};
\node[sommet] (11) at (5,2) {$a$};
\draw (8) --node[left,etiquete] {1} (9);
\draw (9) --node[left,etiquete] {3} (10);
\draw (9) --node[right,etiquete] {4} (11);
\node[etiquete] (t4) at (4.5,-0.6) {$\T_4(\Xb)$};

\node[sommet,gele] (12) at (6.3,0.4) {2};
\node[sommet] (13) at (6.4,1.4) {$a$};
\node[sommet,gele] (14) at (6,2.4) {$5$};
\node[sommet] (15) at (6.8,2.4) {$a$};
\draw (12) --node[left,etiquete] {1} (13);
\draw (13) --node[left,etiquete] {3} (14);
\draw (13) --node[right,etiquete] {4} (15);
\node[etiquete] (t5) at (6.3,-0.2) {$\T_5(\Xb)$};

\end{tikzpicture}
\caption{On the left is represented the walk $(S_n(\Xb))_{n \geq 0}$ up to time $n=5$ associated with the sequence $\Xb=+1,-1,+1,+1,-1,\dots$. On the right, a possible realisation of the trees $\T_0(\Xb)$ to $\T_5(\Xb)$ given this sequence. Frozen vertices have been colored in blue.}
\label{fig exemple arbre recursif avec gel}
\end{center}
\end{figure}

\subsection{Uniform attachment with freezing: growth-coalescent construction}
\label{ssec:growthcoalescent}

In order to analyse the geometry of uniform attachment trees with freezing,  an alternative time-reversed construction based on a growth-coalescence process of forests, that we introduced in \cite{BBKK23+}, is particularly useful. This algorithm is a generalization of the connection between random recursive trees and Kingman's coalescent which first appeared in \cite{DR76} (see \cite[Sec.~6]{Dev87}, \cite[Sec.~3]{Pit94}, \cite[Sec.~2.2]{AB15},  {\cite{ABE18,Esl21,Esl22}} for applications), in the context of union-find data structures (which are data structures that store a collection of disjoint sets where merging sets and finding a representative member of a set), see e.g.~\cite{KS78}.

\paragraph{\labeltext[2]{Algorithm 2.}{algo2}}
Fix $0 \leq n \leq \tau(\Xb)$. We construct a sequence  $(\F_{n}^n(\Xb),\F_{n-1}^n(\Xb), \ldots, \F_{0}^n(\Xb))$ of forests of rooted, edge-labelled, vertex-labelled, unoriented trees by induction as follows.
\begin{compitem}
\item Let $\F_{n}^n(\Xb)$ be a forest made of $S_{n}(\Xb)$ one-vertex trees labelled $a_{1}, \ldots,a_{S_{n}(\Xb)}$.
\item For every $1\leq i \leq n$, if $\F^{n}_{i}(\Xb)$ has been constructed, define $\F^{n}_{i-1}(\Xb)$ as follows:
\begin{compitem}
\item[--] if $\X_{i}=-1$, $\F^n_{i-1}(\Xb)$ is obtained by adding to $\F^n_{i}(\Xb)$ a new one-vertex tree labelled $i$;
\item[--] if $\X_{i}=1$, let $(T_{1},T_{2})$ be a  pair of different random trees in $\F_{i}^n(\Xb)$ chosen uniformly at random, independently of the previous choices; then $\F^n_{i-1}(\Xb)$ is obtained from $\F^n_{i}(\Xb)$ by adding an edge labelled $i$ between the roots $r(T_{1})$ and $r(T_{2})$ of respectively $T_{1}$ and $T_{2}$, and rooting the tree thus obtained at $r(T_{1})$;
\end{compitem}
\item Let $\T^n(\Xb)$ be the only tree of $\F^n_0(\Xb)$.
\end{compitem}

 \begin{figure}[!ht] \centering
\includegraphics[width=1.05\linewidth]{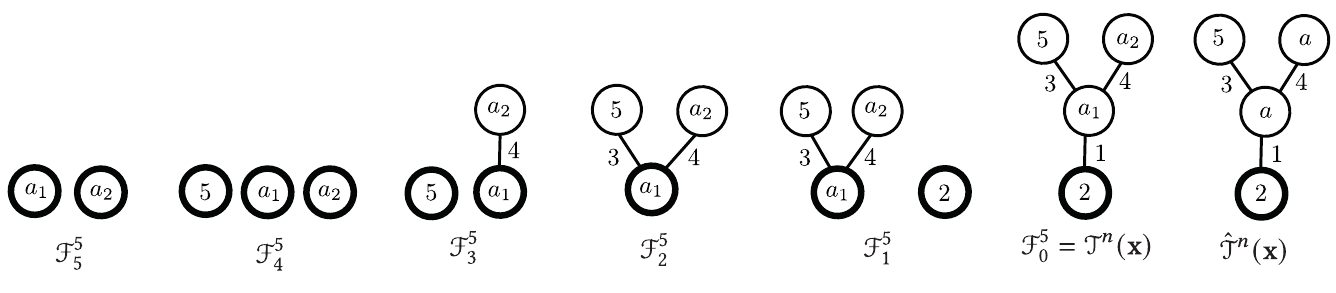}
\caption{An illustration of Algorithm \ref{algo2} with $n=5$ and $(\X_{5},\X_{4},\X_{3},\X_{2},\X_{1})=(-1,1,1,-1,1)$ (this is the same sequence as in Fig.~\ref{fig exemple arbre recursif avec gel}). For example, since $\X_{2}=-1$, $\mathcal{F}^{5}_{1}$ is obtained from $\mathcal{F}^{5}_{2}$ by adding a new tree made of a vertex labeled $2$. Since $\X_{4}=1$, to build  $\mathcal{F}^{5}_{3}$ from  $\mathcal{F}^{5}_{4}$ we have chosen in $\mathcal{F}^{5}_{3}$  two trees  $(T_{1},T_{2})$ with $T_{1}$ being the vertex $a_{1}$ and $T_{2}$ being the vertex $a_{2}$, and  we have added an edge labelled $4$ between the roots $r(T_{1})=a_{1}$ and $r(T_{2})=a_{2}$ of respectively $T_{1}$ and $T_{2}$, and  rooting the tree thus obtained at $r(T_{1})=a_{1}$.}
\label{fig:growthcoalescent}
\end{figure}

\bigskip

Observe that the active vertices of $\T_n(\Xb)$ are all labelled $a$ while the active vertices of $\T^n(\Xb)$ are labelled $a_{1}, \ldots,a_{S_{n}(\Xb)}$ (see Fig.~\ref{fig:growthcoalescent}). It turns out that  $\T_n(\Xb)$ is equal in law to  $\T^n(\Xb)$ when its active vertices are all relabelled $a$. More precisely, denote by $\hat{\T}^n(\Xb)$ the tree obtained from $\T^n(\Xb)$ by relabelling its $S_{n}$ active vertices by $a$ (see the right-most part of Fig.~\ref{fig:growthcoalescent} for an illustration).

\begin{theorem}[Theorem 8 from \cite{BBKK23+}]
\label{thm:samelaw}
The two trees  $\T_n(\Xb)$ and $\hat{\T}^n(\Xb)$ have the same distribution.
\end{theorem}

\subsection{Laws of the birth and coalescence times}
\label{ssec:laws of birth and coal}

We now introduce some notation related to birth and coalescence times, and identify their laws in view of the analysis of scaling limits of trees built by uniform attachment with freezing.

{In the rest of the paper we shall use the notation $[n]= \{1,2, \ldots,n\}$.}  First, define
\begin{equation}
\label{eq:labels}
\Fbb_{n}(\Xb) =  \{i \in [n]: \X_{i}=-1\},  \qquad \A_{n}(\Xb)=  \{a_{1}, \ldots, a_{S_{n}(\Xb)}\}, \qquad   \V_{n}(\Xb)= \Fbb_{n}(\Xb) \cup \A_{n}(\Xb).\end{equation}
Note that while $\T^n(\Xb)$ is a random tree, the labels of its vertices are deterministic as they depend only on $\Xb$: $ \A_{n}(\Xb)$ are the labels of the active vertices of  $\T^n(\Xb)$ and $\Fbb_{n}(\Xb)$ are the labels of the frozen vertices of $\T^{n}(\Xb)$. In particular, the elements of $\V_{n}(\Xb)$ will be called vertices of $\T^n(\Xb)$.

Next, for every $u \in \V_{n}(\Xb)$, we denote by $\birth_n(u)$ the largest $i \in \{0,1,\ldots,n\}$ such that $u$ belongs to the forest $\F^n_i(\Xb)$. Explicitly, if $u \in \A_{n}(\Xb)$ is an active vertex then $\birth_n(v)=n$, and if $u \in \Fbb_{n}(\Xb)$ (note that $u$ is then an integer) then $\birth_n(u)=u-1$ (see Fig.~\ref{fig:growthcoalescent} for an example). We say that $\birth_n(u)$ is the \emph{birth time} of $u$, since it encodes the first time when vertex $u$ appears in Algorithm \ref{algo2}.

For  $0 \leq i \leq \birth_n(u)$, denote by $\H^n_i(u)$ the height of vertex $u$ in $\F^n_i$ (that is, the graph distance between $u$ and the root of $\F^{n}_{i}$).

Finally,  for $u,v \in \V_{n}(\Xb)$, we denote by $\coal_{n}(u,v)$ the largest $i \in \{0,1, \ldots,n\}$ such that $u$ and $v$  belong to the same tree in the forest $\mathcal{F}^{n}_{i}(\Xb)$ obtained when building $\T^{n}(\Xb)$ in Algorithm \ref{algo2}.  We say that $\coal_{n}(u,v)$ is the \emph{coalescence time} between $u$ and $v$, since it encodes the first time $u$ and $v$ belong to the same tree  in Algorithm \ref{algo2} (observe that while $\birth_{n}(u)$ is deterministic, $\coal_{n}(u,v)$ is random).

We now state several simple consequences, which will be useful in the analysis of the geometry of $\T^{n}(\Xb)$ (see Section 2.3 in \cite{BBKK23+} for proofs).
\begin{lemma}
\label{lem:accroissementH}
Fix $1 \leq n \leq \tau(\Xb)$  and  $v \in \V_{n}(\Xb)$. For  every $1 \leq i \leq \birth_{n}(v)$:
\[
\proba{\H^n_{i-1}(v)-\H^n_i(v)=1 | \F^n_n, \ldots, \F^n_i} = \frac{1}{S_{i}}\mathbbm{1}_{\{\X_i=1\}}.
\]
\end{lemma}

 In particular, if $(Y^{n}_i)_{1\leq i \leq n}$ are independent Bernoulli random variables of parameter $1/S_i$, we have for every $u \in \V^{n}$
\begin{equation}
\label{eq:H0} \H^{n}_{0}(u) \quad \mathop{=}^{(d)} \quad \sum_{i=1}^{\birth_{n}(u)} Y^{n}_i\1_{\{\X_i=1\}}.
\end{equation}

\begin{lemma}
\label{lem:bunif}
Fix $1 \leq n \leq \tau(\Xb)$. Let $V$ be an element of $\V_{n}(\Xb)$ chosen uniformly at random. For every $1 \leq m \leq n$ we have
\[
\proba{\birth_n(V) < m}  = \frac{m+1-S_m(\Xb)}{n+1+S_n(\Xb)}.    
\]
\end{lemma}
 The last useful result identifies the law of the coalescence times between two vertices.
\begin{lemma}
\label{loi du temps de coalescence}
Fix $1 \leq n \leq \tau(\Xb)$ and consider $u,v \in \V_{n}(\Xb)$. Then for every $0 \leq c < \birth_{n}(u) \wedge \birth_{n}(v)$ such that $\X_{c+1}=1$:
\[\proba{\coal_{n}(u,v)=c}=\frac{1}{\binom{S_{c+1}}{2}} \prod_{\substack{i=c+2 \\ \text{s.t. } \X_i=1}}^{ \birth_{n}(u) \wedge \birth_{n}(v) } \left(1-\frac{1}{\binom{S_i}{2}}\right).\]
\end{lemma}

\section{Distances in uniform attachment trees with freezing}
\label{sec:distances}

We shall now study the geometry of  uniform attachment trees with freezing. 
 It will be convenient to work with $\T^n(\Xb)$ as built using Algorithm \ref{algo2}. Recall Theorem~\ref{thm:samelaw}: the only difference between $\T_{n}(\Xb)$  and $\T^{n}(\Xb)$ is that all the active vertices of $\T^{n}(\Xb)$ are labelled $a_{1}, \ldots, a_{S_{n}(\Xb)}$, while all the active vertices of $ \T_{n}(\Xb)$ are labelled $a$. In particular the graph structure of both trees is the same in law, so it is equivalent to establish our main results with $\T_{n}(\Xb)$ replaced with $\T^{n}(\Xb)$.

\begin{table}[htbp]\caption{{Table of the main notation and symbols introduced in Section \ref{sec:distances} and used later.}}
\centering
{\begin{tabular}{c c p{12cm} }
\toprule
$d^{n}$ &&  \, graph distance on the set of all vertices $\V_{n}(\Xb)$ of $\T^n(\Xb)$\\
$D^{n}$ &&
\begin{tabular}{l}
       modified distance on $\V_{n}(\Xb)$ defined for $u,v \in  \V_{n}(\Xb)$ by \\
$\displaystyle \dc^n(u,v) =  \frac{1}{2}\sum_{\coal_n(u,v) \leqslant  i  \leqslant \birth_n(u)}\frac{1}{S^n_i}+ \frac{1}{2}\sum_{\coal_n(u,v) \leqslant  i \leqslant \birth_{n}(v)}\frac{1}{S^n_i}$
    \end{tabular} \\
\bottomrule
\end{tabular}
}\label{tab:secdistances}
\end{table}

Also recall from \eqref{eq:labels} the definition of $\mathbb{V}_{n}(\Xb)$, which is a deterministic set representing the labels of the vertices of $\T^{n}(\Xb)$, and that by a slight abuse of notation we view  elements of $\mathbb{V}_{n}(\Xb)$ as  vertices of  $\T^{n}(\Xb)$.

As in \cite{BBKK23+}, we will use Bennett's inequality many times, under the {following variant} tailored for our purpose:
\begin{proposition}[Bennett's inequality]
\label{prop:bennett}
Let $(Y_{i})_{1 \leq i \leq n}$ be independent random variables, such that $Y_{i}$ follows the Bernoulli distribution of parameter $p_{i} \in [0,1]$.  Set $m_{n}=\sum_{i=1}^{n} p_{i}$. Assume $m_{n}>0$. Then for every $t>0$:
\[\proba{\sum_{i=1}^{n} Y_{i} >m_{n}+ t} \leq \exp \left( - m_{n} g \left( \frac{t}{m_{n}} \right) \right) \qquad \textrm{and} \qquad \proba{\sum_{i=1}^{n} Y_{i}< m_{n}- t} \leq \exp \left( - m_{n} g \left( \frac{t}{m_{n}} \right) \right),\]
 where $g(u)=(1+u)\ln(1+u)-u$ for $u>0$.
\end{proposition}

\subsection{Deterministic estimates}
\label{ssec:estimates}

\begin{center}

\noindent\rule{0.8\textwidth}{1pt}

From now on, we consider a sequence $(\Xb^n)_{n \geq 1}$ of sequences of elements of $\{-1,1\}$ such that $\tau(\Xb^n)>n$ for every $n \geq 1$.

\noindent\rule{0.8\textwidth}{1pt}
\end{center}

To simplify notation, in the sequel, we drop the dependence in $\Xb^n$, so for instance we will write $\T^n$ for the tree $\T^n(\Xb^n)$ built from $\Xb^{n}$ using Algorithm \ref{algo2}. We also set $S_k^n \coloneqq S_k(\Xb^n)$ for all $k\leq n$, and define  for every $1\leq a \leq b < \tau(\Xb^{n})$:
\begin{equation}
\label{eq:defhn}
\h_{n}(a,b) \coloneqq \sum_{i=a}^b \frac{1}{S^{n}_i},  \quad \h_{n} \coloneqq \h_{n}(1,n), \qquad \h_{n}^{+}(a,b)  \coloneqq \sum_{i=a}^b \frac{1}{S^{n}_i} \mathbbm 1_{\{\X^{n}_i=1\}}, \quad  \h_{n}^{+} \coloneqq h_{n}^{+}(1,n).
\end{equation}

Our goal is now to study the geometry of $\T^n$.
To this end, it is useful to estimate $\h_{n}^{+}${, since it is  connected with the evolution of distances in the growth-coalescent construction (see Lemma \ref{lem:accroissementH}). The following result will allow to replace $\h_{n}^{+}$ with $\h_{n}$ up to a factor $2$ (this comes frome the fact that  in the regime under consideration there are roughly as much $+1$'s as $-1$'s and that their locations are roughly uniformly distributed), and in turn the asymptotic behavior of $\h_{n}$ can be related to the integral of $1/f$ by \ref{hyp:conv1surf}.}

\begin{lemma}
\label{lem:hnalpha}
Let $\alpha \in (0,1)$ and $f:[0,1] \rightarrow \R_+$. Assume that \ref{hyp:conv1surf} and \ref{hyp:conv1surf2} are in force. Then the following assertions hold:
\begin{enumerate}[noitemsep,nolistsep]
\item[(i)] For all $a \in (0,1]$, we have $S^n_{n a} = o(n)$.
\item[(ii)]  
The following convergences
\[
\frac{2}{n} \sum_{i=1}^n \1_{\X_i^n=1} \delta_{i/n} \mathop{\longrightarrow}\limits_{n\to  \infty}
\mathrm{Leb}_{[0,1]} \qquad \text{and} \qquad \frac{2}{n} \sum_{i=1}^n \1_{\X_i^n=-1} \delta_{i/n} \mathop{\longrightarrow}\limits_{n\to  \infty}
\mathrm{Leb}_{[0,1]} 
\]
hold weakly, where $\mathrm{Leb}_{[0,1]}$ is the Lebesgue measure on $[0,1]$.
\item[(iii)]  
\[
2n^{\alpha-1}\h_n^+ \sim n^{\alpha-1}\h_n \sim  \int_{0}^{1} \frac{1}{f(t)}\mathrm{d}t.
\]
\item[(iv)]
\[
\frac{1}{n^{1-\alpha}} \max_{1 \leq a \leq b \leq n} \left| \h^{+}_n(a,b)- \frac{1}{2} \h_n(a,b) \right|  \quad \mathop{\longrightarrow}_{n \rightarrow \infty} \quad 0.
\]
\item[(v)]
For all $\e>0$,
\[
\max_{\e n \le a \le b \le (1-\e)n}
\left\vert
\frac{1}{n} \sum_{i=a}^b \frac{n^{2\alpha}}{S^n_i(S^n_i-1)} \1_{\X^n_i=1}
-\int_{a/n}^{b/n} \frac{\mathrm{d}t}{2 f(t)^2}
\right\vert
\quad
\mathop{\longrightarrow}\limits_{n\to \infty}
0
\]
\end{enumerate}
\end{lemma}

\begin{proof}
For (i), let $a \in (0,1]$. We assume by contradiction that there exists $\e>0$ such that $S_{an}^n\ge \e n$ for infinitely many $n$. Then, since the steps are $\pm 1$, we have for infinitely many $n$, for all $k \le (\e \wedge a) n$, 
$$
S^n_{an-k} \ge S^n_{an}-k \ge \e n -k.
$$
Hence, by summing over $(\e \wedge a) n /2 \le k \le (\e \wedge a) n$, for infinitely many $n$ we have
$$
\sum_{(\e \wedge a) n /2 \le k \le (\e \wedge a) n} \frac{1}{S^n_{an-k}} = O(1),
$$
which is absurd by \ref{hyp:conv1surf}. Thus (i) holds.

For $0 \leq a < b \leq 1$,
\begin{equation}
\label{eq:pn} \#\{\lfloor an\rfloor+1\leq i \leq \lfloor bn \rfloor : \X_i=1\}  = \frac{\lfloor bn \rfloor - \lfloor an \rfloor + S^{n}_{\lfloor bn \rfloor} - S^{n}_{\lfloor an \rfloor}}{2}  \quad \mathop{\sim}_{n \rightarrow \infty} \quad  
\frac{(b-a)n}{2}
\end{equation}
since $S^{n}_{\lfloor an \rfloor}=o(n)$ and $S^{n}_{\lfloor bn \rfloor} =o(n)$, hence (ii) for the $+1$ steps by the Portmanteau theorem. The same reasoning works for the $-1$ steps.

Let us show (iii). The second equivalence is obvious from \ref{hyp:conv1surf}. So, we want to show that 
$
2n^{\alpha-1} \h_n^+ \mathop{\sim}\limits_{n\to \infty} n^{\alpha-1} \h_n$. 
But we have
$$
\h_n^+ + \sum_{i=1}^n \frac{1}{S_i^n} \1_{\X^n_i=-1} = \h_n.
$$
Therefore, it is enough to show that
$$
\h_n^+ \mathop{\sim}\limits_{n\to \infty} \sum_{i=1}^n \frac{1}{S_i^n}\1_{\X^n_i=-1} .
$$
\begin{figure}
	\centering
	\includegraphics[width=0.85\linewidth]{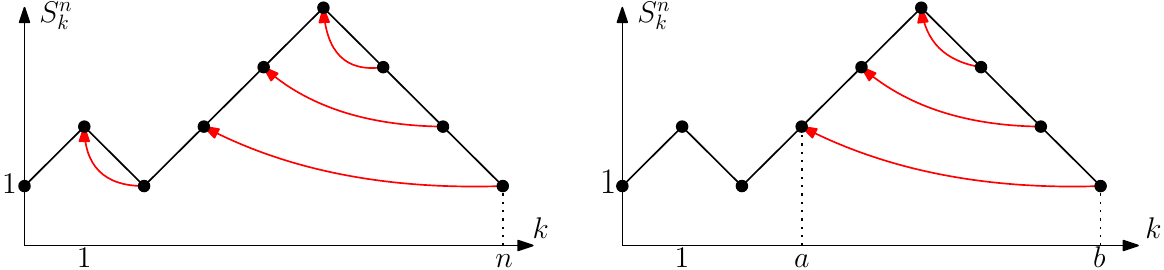}
	\caption{Illustration of the pairings which are used in the proof of the assertions (iii), on the left, and (iv),(v), on the right, of Lemma \ref{lem:hnalpha}.}
	\label{fig:identifications}
\end{figure}
However, by mapping each $-1$ step $\X_i^n=-1$ to the last time $j<i$ such that $S^n_{j-1}=S^n_i $, which is thus such that $\X^n_j = +1$ (see Figure \ref{fig:identifications}), we see that
$$
\h_n^+ = \sum_{i=1}^n \frac{1}{S^n_i +1} \1_{\X^n_i=-1}  + \sum_{s=2}^{S^n_n} \frac{1}{s} 
=  \sum_{i=1}^n \frac{1}{S^n_i } \1_{\X^n_i=-1} -  \sum_{i=1}^n \frac{1}{(S^n_i +1)S^n_i} \1_{\X^n_i=-1} + \ln(S^n_n) +O(1)
= \sum_{i=1}^n \frac{1}{S^n_i } \1_{\X^n_i=-1} + o(n^{1-\alpha}),
$$
where the last equality comes from the fact that $S^n_n \le n+1$ and from the convergence 
\begin{equation}\label{eq:lemme hnalpha convergence somme vers 0}
	\frac{1}{n}\sum_{i=1}^n \frac{n^\alpha}{(S^n_i +1)S^n_i}
	\mathop{\longrightarrow}\limits_{n \to \infty}
	0.
\end{equation}
The above convergence can be shown as follows: if $U$ is a uniform random variable in $[0,1]$, then the sequence $(n^\alpha/(S^n_{Un}(S^n_{Un}+1)))_{n\ge 1}$ is uniformly integrable since it is upper bounded by $(n^\alpha/S^n_{Un})_{n\ge 1}$, which converges in $\Lp^1$ by \ref{hyp:conv1surf}, and $n^\alpha/(S^n_{Un}(S^n_{Un}+1))$ converges in probability towards zero, again by \ref{hyp:conv1surf}. This proves (iii).

For (iv), let $1\le a \le b \le n$. We use the same idea: by mapping each $-1$ step $\X_i^n=-1$ for $a \le i \le b$ to the last time $j$ for $a \le j<i$ such that $S^n_{j-1}=S^n_i $ when it exists (see Figure \ref{fig:identifications}), one can see that
$$
\h_n^+(a,b) = \sum_{i=a}^b \frac{1}{S^n_i+1} \1_{\X^n_i = -1} - \sum_{s=\min_{a \le j \le b} S^n_j}^{S^n_a-1} \frac{1}{s+1} + \sum_{s=\min_{a \le j \le b} S^n_j+1}^{S^n_b} \frac{1}{s}, 
$$
so that for all $1 \le a \le b \le n$
$$
\left\vert
\h^+_n(a,b)- \sum_{i=a}^b \frac{1}{S^n_i} \1_{\X^n_i=-1}
\right\vert
\le 
\sum_{i=1}^n \frac{1}{S^n_i(S^n_i+1)} \1_{\X^n_i=-1}
+ 2 \sum_{s=1}^{n+1} \frac{1}{s}
= o(n^{1-\alpha}),
$$
where the $o(n^{1-\alpha})$ is uniform in $1\le a \le b \le n$ and comes from (\ref{eq:lemme hnalpha convergence somme vers 0}). But
$$
\h_n^+(a,b)-\frac{1}{2} \h_n(a,b) = \frac{1}{2} \left( 
\h_n^+(a,b) - \sum_{i=a}^b \frac{1}{S^n_i} \1_{\X^n_i=-1}
\right),
$$
hence (iv).

For the last point (v), let $\e \in (0,1)$. We use the same trick, so that for all $1 \le a \le b \le n$,{
\begin{equation}\label{eq:lemme h point cinq1}
\sum_{i=a}^b \frac{1}{S^n_i (S^n_i-1)} \1_{\X^n_i=1}
=
\sum_{i=a}^b \frac{1}{(S^n_i+1)S^n_i} \1_{\X^n_i = -1} +O\left( \frac{1}{\min_{\e n \le k \le (1-\e)n} S^n_k}\right),
\end{equation}
where the $O( 1/{\min_{\e n \le k \le (1-\e)n} S^n_k})$ is uniform in $\e n\le a \le b \le (1-\e)n$.}

{
Moreover, note that the condition \ref{hyp:conv1surf2} implies that
\begin{equation}\label{eq:cvminSinfini}
	\frac{1}{\min_{\e n \le k \le (1-\e)n} S^n_k }= o \left(n^{1-2\alpha}\right).
\end{equation}
Indeed, let $(k_n)_{n\ge 0}$ be a sequence of integers such that for all $n\ge 0$, we have $S^n_{k_n} = \min_{\e n \le k \le (1-\e)n} S^n_k$. Let $\eta_n>0$ such that $\eta_n \to 0$ and $S^n_{(1-\e)n}=o(n\eta_n)$ (such a sequence exists thanks to point (i)). Then we have as $n\to \infty$,
$$
\int_{k_n/n }^{k_n/n+\eta_n} \left(
\frac{n^\alpha}{S^n_{nt}}
\right)^2 dt  \ge \frac{1}{n} \sum_{j=1}^{\lfloor \eta_n n \rfloor -1} \frac{n^{2\alpha}}{(\min_{\e n \le k \le (1-\e)n} S^n_k +j)^2}
\ge (1+o(1)) \frac{n^{2\alpha-1}}{\min_{\e n \le k \le (1-\e)n} S^n_k+1}.
$$
By \ref{hyp:conv1surf2}, the integral on the left tends to zero, hence \eqref{eq:cvminSinfini}.
}

{
Besides, using \ref{hyp:conv1surf2} and the same idea as for \eqref{eq:lemme hnalpha convergence somme vers 0}, one can check that
\begin{equation}\label{eq:lemme h point cinq2}
\frac{1}{n}\sum_{i=a}^b \frac{n^{2\alpha}}{S^n_i (S^n_i-1)} \1_{\X^n_i=1} - \frac{1}{n}\sum_{i=a}^b \frac{n^{2\alpha}}{S^n_i (S^n_i+1)} \1_{\X^n_i=1} = o(1).
\end{equation}
} Thus, combining \eqref{eq:lemme h point cinq1}, \eqref{eq:cvminSinfini} and \eqref{eq:lemme h point cinq2}, 
\begin{align*}
\max_{\e n \le a \le b \le (1-\e)n}
&\left\vert
\frac{1}{n} \sum_{i=a}^b \frac{n^{2\alpha}}{S^n_i(S^n_i-1)} \1_{\X^n_i=1}
-\int_{a/n}^{b/n} \frac{\mathrm{d}t}{2 f(t)^2}
\right\vert
\\
&\le o(1) + 
\max_{\e n \le a \le b \le (1-\e)n}
\left\vert
\frac{1}{n} \sum_{i=a}^b \frac{n^{2\alpha}}{2S^n_i(S^n_i+1)} 
-\int_{a/n}^{b/n} \frac{\mathrm{d}t}{2 f(t)^2}
\right\vert
\quad
\mathop{\longrightarrow}\limits_{n\to \infty}
0,
\end{align*}
where the convergence is due to \ref{hyp:conv1surf2}.
\end{proof}
\subsection{An alternative distance using  coalescence times}
\label{ssec:newdistance}

{Recall that $d^{n}$ denotes the graph distance on the vertices $\V_{n}$ of the tree $\T_n$.}  In this section, we introduce a new distance on {$\V_{n}$}  which will be more convenient to study, based on the coalescence time {between vertices}.  For every $n\in \N$, and $u,v\ \in \mathcal{T}^{n}$, set
\[ \dc^n(u,v) \coloneqq  \frac{1}{2}\sum_{\coal_n(u,v) \leqslant  i  \leqslant \birth_n(u)}\frac{1}{S^n_i}+ \frac{1}{2}\sum_{\coal_n(u,v) \leqslant  i \leqslant \birth_{n}(v)}\frac{1}{S^n_i} = \frac{1}{2} \h_n(\coal_{n}(u,v),b_{n}(u))+ \frac{1}{2}\h_n(\coal_{n}(u,v),b_{n}(v)),\]
where we use the notation $\h$ introduced in \eqref{eq:defhn}.
\begin{proposition}
\label{height=ok}
Let $\alpha \in (0,1)$ and $f:[0,1] \rightarrow \R_+$. Assume that \ref{hyp:conv1surf} and \ref{hyp:conv1surf2} are in force. Then for every $\e>0$,
\[
\P \left (\max_{u,v\in \V_{n}} |d^n(u,v)-\dc^n(u,v)|>\e n^{1-\alpha} \right )  \quad \mathop{\longrightarrow}_{n \rightarrow \infty} \quad  0.
\]
\end{proposition}

\begin{proof}
Recall  that, for $v \in \mathbb{V}_{n}$ and $0 \leq i \leq \birth_{n}(v)$,  $\H^{n}_{i}(v)$ is the height of vertex $v$ in $\mathcal{F}^{n}_{i}$. 
First note that by the definition of Algorithm \ref{algo2},  for every $u,v\in \V_{n}$, the nearest common ancestor of $u$ and $v$ in $\T^{n}$ is the root of the tree in the forest $\mathcal{F}^{n}_{\coal_{n}(u,v)}$ which contains $u$ and $v$. As a result, $d^{n}(u,v)=\H^{n}_{\coal_{n}(u,v)}(u)+\H^{n}_{\coal_{n}(u,v)}(v)$.
So
\[ |d^{n}(u,v)-\dc^{n}(u,v)|= \left  | \H^{n}_{\coal_{n}(u,v)}(u)+\H^{n}_{\coal_{n}(u,v)}(v)- \frac{1}{2} \h_n(\coal_{n}(u,v),b_{n}(u))- \frac{1}{2} \h_n(\coal_{n}(u,v),b_{n}(v)) \right |. \]
Hence
\begin{equation}
\label{4/9/10h}
\max_{u,v\in \V_{n}} |d^{n}(u,v)-\dc^{n}(u,v)| \leq 2\max_{\substack{v\in \V_{n} \\ 0 \leq a<\birth_{n}(v)}} \left  |\H^{n}_{a}(v)- \frac{1}{2}  \h_{n}(a,b_{n}(v))\right |. 
\end{equation}

We next estimate the right hand side of \eqref{4/9/10h}. It follows from Lemma \ref{lem:accroissementH} that for every $v\in \V_{n}$ and $ 0 \leq a \leq \birth_{n}(v)$, 
\[ \H^{n}_a(v) \quad \mathop{=}^{(d)} \quad \sum_{i=a+1} ^{\birth_{n}(v)}  Y_{i}^{n} \1_{\X^{n}_i=1},\]
where $(Y_i^n)_{1 \leq i \leq n}$  are independent Bernoulli random variables of {respective} parameters $(1/S^n_i)_{1 \leq i \leq n}$.
As before, set $g(u)=(1+u)\log(1+u)-u$ for $u>0$. Then, by Bennett's inequality, recalled in Proposition \ref{prop:bennett}, and using the fact that $x \mapsto x g(t/x)$ is decreasing, for every $t>0$, we see that $\proba{ \left|  \H^{n}_a(v) -  \h^{+}_{n}(a+1,b_{n}(v)) \right|>t}$ {is smaller than} $2\exp(-h^{+}_n g(t/h^{+}_n))$.
Therefore, by a union bound,
\begin{equation}
\proba{\max_{{v\in \V_{n}, 0 \leq a<\birth_{n}(v)}} \left  |\H^{n}_{a}(v)-   \h^{+}_{n}(a+1,b_{n}(v))\right |>t}\leq 2n^2\exp(-h^{+}_n g(t/h^{+}_n)).
\label{4/9/14h}
\end{equation}
By Lemma \ref{lem:hnalpha} (iii)  it follows that  for every $\e>0$,
\[ \P\left (\max_{v\in \V_{n}, 0 \leq a<\birth_{n}(v)} \left  |\H^{n}_{a}(v)- \h^{+}_{n}(a+1,b_{n}(v)) \right |> \e n^{1-\alpha} \right )\leq 2n^2\exp(-h^{+}_n g(\e n^{1-\alpha}/h^{+}_n))\to 0.\] 
{In addition}, by Lemma \ref{lem:hnalpha} (iv), and since $ {1}/{S_{i}} \leq 1$ for every $0 \leq i \leq n$, we have
\[
\frac{1}{n^{1-\alpha}} \max_{1 \leq a \leq b \leq n} \left| \h^{+}_n(a+1,b)- \frac{1}{2} \h_n(a,b) \right|  \quad \mathop{\longrightarrow}_{n \rightarrow \infty} \quad 0,
\]
which completes the proof.
\end{proof}

\section{Scaling limits: non-critical regimes}
\label{sec:noncritical}

As before, we consider a sequence $(\Xb^n)_{n \geq 1}$ of sequences of elements of $\{-1,1\}$ such that $\tau(\Xb^n)>n$ for every $n \geq 1$.

\subsection{Subcritical  regime \texorpdfstring{$\alpha \in (0,1/2)$}{}}

We first show that, roughly speaking, vertices coalesce quickly.

\begin{lemma}\label{lemme ça coalesce vite}
Let $\alpha \in (0,1/2)$. Let $\alpha \in (0,1)$ and $f:[0,1] \rightarrow \R_+$. Assume that \ref{hyp:conv1surf} and \ref{hyp:conv1surf2} are in force. Then, 
\[
\frac{1}{n}\max_{u,v \in \V_{n}}\left| \birth_{n}(u) \wedge \birth_n(v)-\coal_{n}(u,v)\right| \cvproba[n] 0.
\]
\end{lemma}
\begin{proof}
{Recall that the total number of vertices of $ \mathcal{T}^{n}$ is $(n+S_n^n+1)/2$.} {Using Lemma \ref{loi du temps de coalescence}}, by a union bound and then by \ref{hyp:conv1surf2}, for $n$ large enough
\begin{align*}
\P\left( \exists u, v \in \V_{n}: \left| \birth_n(u) \wedge \birth_{n}(v)-\coal_n(u,v)\right| \ge \e n\right) &\le \left(\frac{n+S_n^n+1}{2}\right)^2
\max_{\lfloor\e n \rfloor+1\le b \le n}\prod_{\substack{i=b-\lfloor \e n \rfloor+1 \\ \text{s.t. } \X^n_i=1}}^b \left( 1-\frac{1}{\binom{S_i^n}{2}}\right) \\
&\le(n+1)^2 \exp\left(\max_{\lfloor \e n \rfloor +1 \le b \le n} \sum_{i=b-\lfloor \e n \rfloor+1}^b \frac{-2}{S^n_i(S^n_{i}-1)}\1_{\X^n_i=1} \right) \\
&\le (n+1)^2
e^{-(n^{1-2\alpha}/2)\min{\e \le y \le 1} \int_{y-\e/2}^y {\mathrm{d} t}/{f(t)^2} }
\cv[n] 0,
\end{align*}
where the last inequality comes from Lemma \ref{lem:hnalpha} (v).
This completes the proof.
\end{proof}

\begin{proof}[Proof of Theorem \ref{thm:sous-critique}]
Let $k\ge 1$. For all $n\ge 1$, let $V^n_1,\ldots,V^n_k$ be uniform random vertices of $\T^n$. Let $U_1,\ldots, U_k$ be i.i.d. uniform random variables in $[0,1]$. Let us first show that 
\[
\left(\frac{d^n(V^n_j,V^n_\ell)}{n^{1-\alpha}}\right)_{1\le j,\ell\le k} \cvloi[n] \left(\int_{U_{j}\wedge U_\ell}^{U_j \vee U_\ell} \frac{\mathrm{d}t}{2f(t)}\right)_{1\le j,\ell\le k}.
\]
By Proposition \ref{height=ok}, it suffices to prove that
\begin{equation}\label{eq thm sous-critique}
\left(\frac{1}{n^{1-\alpha}}\left(\sum_{i=\coal_n(V^n_j,V^n_\ell)}^{\birth_n(V^n_j)} \frac{1}{2S^n_i} 
+\sum_{i=\coal_n(V^n_j,V^n_\ell)}^{\birth_n(V^n_\ell)} \frac{1}{2S^n_i}
\right)\right)_{1\le j,\ell\le k} \cvloi[n] \left(\int_{U_{j}\wedge U_\ell}^{U_j \vee U_\ell} \frac{\mathrm{d}t}{2f(t)}\right)_{1\le j,\ell\le k}.
\end{equation}
However, one can write
\[
\sum_{i=\coal_n(V^n_j,V^n_\ell)}^{\birth_n(V^n_j)} \frac{1}{2S^n_i} 
+\sum_{i=\coal_n(V^n_j,V^n_\ell)}^{\birth_n(V^n_\ell)} \frac{1}{2S^n_i}
=\sum_{i=\coal_n(V^n_j,V^n_\ell)}^{\birth_n(V^n_j)\wedge \birth_n(V^n_\ell)} \frac{1}{S^n_i}
+\sum_{(\birth_n(V^n_j)\wedge \birth_n(V^n_\ell))+1}^{\birth_n(V^n_j)\vee\birth_n(V^n_\ell)} \frac{1}{2S^n_i}.
\]
Besides, $S^n_n=o(n)$, so that with high probability the vertices $V^n_j$'s are frozen. Moreover, conditionally on being frozen, the $\birth_n(V^n_j)$'s are independent and uniform in $\left\{i \in [n]; \X_i^n=-1\right\}$. Therefore, by Lemma \ref{lem:hnalpha} (ii), we know that the $\birth_n(V^n_j)/n$'s converge in distribution towards the $U_j$'s. Finally, combining this remark with Lemma \ref{lemme ça coalesce vite} and \ref{hyp:conv1surf}, we obtain (\ref{eq thm sous-critique}).

So as to end the proof of the GHP convergence, by {Lemma \ref{GP+GH=GHP}}, it suffices to establish the convergence in the sense Gromov-Hausdorff since we already have the Gromov-Prokhorov convergence. Let us define the correspondence
\[
\mathcal{C}_n=\left\{\left(u, \int_0^t \frac{1}{2f(s)} \mathrm{d}s\right); \ u \in \V_n \text{ and }\left|t-\frac{\birth_n(u)}{n}\right| = \min_{v\in \V_n} \left|t-\frac{\birth_n(v)}{n}\right|\right\}.
\]
By Theorem 7.3.25 in \cite{BBI01} it then suffices to show that
\begin{equation}\label{cv distorsion}
\mathrm{dis} (\mathcal{C}_n) \coloneqq \sup_{(u, \int_0^t \frac{1}{2f(s)} \mathrm{d}s), (v, \int_0^r \frac{1}{2f(s)} \mathrm{d}s) \in \mathcal{C}_n}
\left|
\frac{d^n(u,v)}{n^{1-\alpha}} - \int_{t\wedge r}^{t\vee r}\frac{1}{2f(s)} \mathrm{d}s
\right|
\cvproba[n] 0.
\end{equation}
We upper bound
\begin{align}
\mathrm{dis}(\mathcal{C}_n)
\le &\sup_{u, v \in \V_n}\frac{ |d^n(u,v)-\dc^n(u,v)|}{n^{1-\alpha}} \label{première ligne distorsion} \\
&+ \sup_{u,v \in \V_n}\left|
\frac{1}{n^{1-\alpha}} \sum_{i=\coal_n(u,v)}^{\birth_n(u)\wedge \birth_n(v)} \frac{1}{S_i^n}
\right| \label{deuxième ligne distorsion}\\
&+\sup_{(u, \int_0^t \frac{1}{2f(s)} \mathrm{d}s), (v, \int_0^r \frac{1}{2f(s)} \mathrm{d}s) \in \mathcal{C}_n}
\left|
\frac{1}{n^{1-\alpha}}
\sum_{i=\birth_n(u)\wedge \birth_n(v) +1}^{\birth_n(u) \vee \birth_n (v)} \frac{1}{2S^n_i} - \int_{t\wedge r}^{t\vee r} \frac{1}{2f(s)} \mathrm{d}s
\right| \label{troisième ligne distorsion}
\end{align}

The first line (\ref{première ligne distorsion}) converges to zero in probability as $n\to \infty$ thanks to Proposition \ref{height=ok}. For the second line (\ref{deuxième ligne distorsion}),
\[
\sup_{u,v \in \V_n}\left|
\frac{1}{n^{1-\alpha}} \sum_{i=\coal_n(u,v)}^{\birth_n(u)\wedge \birth_n(v)} \frac{1}{S_i^n}
\right| 
\le 
\int_0^1 \left\vert
\frac{n^\alpha }{S^n_{\lfloor ns \rfloor}} -\frac{1}{f(s)} 
\right\vert
\mathrm{d} s
+
\sup_{u, v \in \V_n} \int_{\coal_n(u,v)/n}^{(\birth_n(u)\wedge \birth_n(v)+1)/n} \frac{\mathrm{d}s}{f(s)}
\cvproba[n] 0
,
\]
by \ref{hyp:conv1surf} and by Lemma \ref{lemme ça coalesce vite}. Thus the second line (\ref{deuxième ligne distorsion}) converges to zero in probability as $n\to \infty$. 

We can upper-bound the quantity appearing in (\ref{troisième ligne distorsion}) by
$$
\int_0^1 \left|
\frac{n^\alpha}{2S^n_{\lfloor ns \rfloor}} - \frac{1}{2f(s)} 
\right| \mathrm{d}s +
 \sup_{(u, \int_0^t \frac{1}{2f(s)} \mathrm{d}s), (v, \int_0^r \frac{1}{2f(s)} \mathrm{d}s) \in \mathcal{C}_n}
 \left| \int_{ \frac{\birth_n(u)\wedge \birth_n(v) +1}{n}}^{ \frac{\birth_n(u) \vee \birth_n (v)+1}{n}} \frac{\mathrm{d}s}{2 f(s)}- \int_{t\wedge r}^{t\vee r} \frac{\mathrm{d}s}{2f(s)} 
\right|.
$$
Observe that, by definition of $\mathcal{C}_n$ and thanks to Lemma \ref{lem:hnalpha} (ii), we have
\[
\sup_{(u, \int_0^t \frac{1}{2f(s)} ds), (v, \int_0^r \frac{1}{2f(s)} ds) \in \mathcal{C}_n} \left( \left| t \wedge r- \frac{\birth_n(u)\wedge \birth_n(v)}{n} \right| +  \left| t \vee r- \frac{\birth_n(u)\vee \birth_n(v)}{n} \right| \right)
\cvproba[n] 0,
\]
and combined with the integrability of $1/f$ this completes the proof.
\end{proof}

We complete the study of the subcritical regime by establishing a scaling limit result when the sequence  $\| S^{n} \|_{\infty}$ is bounded (which roughly speaking may be thought of {as} the case $\alpha=0$). The main difference with Theorem \ref{thm:sous-critique} is that now the limit is for the Gromov--Hausdorff topology, and not in the Gromov--Hausdorff--Prokhorov sense (in general there is no universal behavior for the latter topology).

\begin{theorem}
\label{thm:alpha0}
Let $(\Xb^n)_n$ be such that $\tau(\Xb^n)>n$ for all $n$ and $S_k^n \leq M$ for all $k\leq n$ where $M$ is some positive constant. Then
\[
\frac{1}{\h_n^+} \cdot \T_{n}  \quad \mathop{\longrightarrow}^{\P}_{n \rightarrow \infty} \quad \left[0,1\right]
\]
holds in probability for the Gromov-Hausdorff topology.
\end{theorem}

\begin{proof}
For all $n$, let us call a \emph{side sub-tree} of $\T_n$, a sub-tree rooted at a vertex of the longest branch of $\T_n$ and containing only vertices that are not on the longest branch (except the root). The height of the longest branch of $\T_n$ is of order $\h_n^+$. Indeed, it is easy to check that $\h_n^+ = \Theta(n)$. Therefore, by \cite[Theorem 3 (3)]{BBKK23+} we deduce that $\Height(\T_n)/\h_n^+$ converges to $1$ in $\Lp^p$ for all $p\geq 1$. To show the proposition, it is enough to show that, for all $\e>0$, the probability to see a side sub-tree of height larger than $\e n$ in $\T_n$ goes to $0$ when $n$ goes to infinity.

{Since for $\e$ small enough, the longest branch has a length at least $\e n$ with high probability, it suffices to show that the probability that ``there exist two vertices $u,v \in \V_n$ such that their nearest common ancestor $w$ is at distance at least $\e n$ from each of them'' goes to zero as $n\to \infty$. But for all $u,v \in \V_n$, if $w$ is their common ancestor, then one can directly upperbound
$$
d^n(u,w)\le  \birth_n(u)-\coal_n(u,v) \qquad \text{and} \qquad
d^n(v,w) \le \birth_n(v)-\coal_n(u,v).
$$
Finally, by a union bound on $u,v \in \V_n$ and using Lemma \ref{loi du temps de coalescence},
\begin{align*}
\P\left( \exists u, v \in \V_n, \ 
\birth_n(u)\wedge \birth_n(v) -\coal_n(u,v) \ge \e n
\right)
&\le
\left(\frac{n+S_n^n+1}{2}\right)^2
\max_{\lfloor\e n \rfloor+1\le b \le n}\prod_{\substack{i=b-\lfloor \e n \rfloor+1 \\ \text{s.t. } \X^n_i=1}}^b \left( 1-\frac{1}{\binom{S_i^n}{2}}\right) \\
&\le (n+1)^2 \left( 1-\frac{1}{\binom{M}{2}}\right)^{\lfloor \e n\rfloor/2 - M}
\mathop{\longrightarrow}\limits_{n\to \infty} 0,
\end{align*}
where the last inequality stems from the fact that $S^n_k\le M$ for all $k\le n$ and from \eqref{eq:pn} which also holds in this case.
}
\end{proof}

\subsection{Supercritical regime \texorpdfstring{$\alpha \in (1/2,1)$}{}}

We first show that, roughly speaking, vertices coalesce near the origin.

\begin{lemma}
\label{lemme ça coalesce en zéro}
Let $\alpha \in (1/2,1)$ and $f:[0,1] \rightarrow \R_+$. Assume that \ref{hyp:conv1surf} and \ref{hyp:conv1surf2} are in force. For all $n \ge 1$, let $V^n_1, V_2^n$ be two independent uniform vertices of $\V_{n}$. Then
\[
\frac{\coal_{n}(V^{n}_{1},V^{n}_{2})}{n}\cvproba[n] 0.
\]
\end{lemma}

\begin{proof}
Let $\e,\delta \in(0,1)$ with $\e<1-\delta$. By Lemma \ref{loi du temps de coalescence},
\[
\proba{\coal_{n}(V^{n}_{1},V^{n}_{2}) \ge \e n}\leq\proba{\birth_{n}(V^{n}_{1}) \wedge \birth_{n}(V^{n}_{2}) \ge (1-\delta)n} + 
\sum_{\substack{\e n \leqslant i \leqslant (1-\delta)n \\ \X^n_i=1}} \frac{1}{\binom{S_i^n}{2}}.
\]
By Lemma \ref{lem:bunif} and the fact that $S_{n}^n = o(n)$,
\[
\proba{\birth_{n}(V^{n}_{1}) \wedge \birth_{n}(V^{n}_{2}) \ge (1-\delta)n} \mathop{\longrightarrow}\limits_{n \to \infty} \delta^2.
\]
The remaining sum is a $O(n^{1-2\alpha})$ by Lemma \ref{lem:hnalpha} (v).  Making $\delta\to0$ we conclude that $\proba{\coal_{n}(V^{n}_{1},V^{n}_{2}) \ge \e n}\to0$ when $n\to\infty$.
\end{proof}

\begin{proof}[Proof of Theorem \ref{thm:sur-critique}]
By Proposition \ref{height=ok} it suffices to show that
\begin{equation}\label{eq1 thm sur-critique}
\frac{1}{n^{1-\alpha}}\sum_{i=1}^{\birth_n(V_1^n)} \frac{1}{2S^n_i} 
\cvloi[n] 
\int_{0}^{U_1} \frac{\mathrm{d}t}{2f(t)}
\end{equation}
and that
\begin{equation}\label{eq2 thm sur-critique}
\left(\frac{1}{n^{1-\alpha}}\left(\sum_{i=\coal_n(V^n_j,V^n_\ell)}^{\birth_n(V^n_j)} \frac{1}{2S^n_i} 
+\sum_{i=\coal_n(V^n_j,V^n_\ell)}^{\birth_n(V^n_\ell)} \frac{1}{2S^n_i}
\right)\right)_{1\le j,\ell\le k} \cvloi[n] \left(\int_0^{U_j } \frac{\mathrm{d}t}{2f(t)}+\int_0^{U_\ell } \frac{\mathrm{d}t}{2f(t)}\right)_{1\le j,\ell\le k}.
\end{equation}
We know that $S^n_n=o(n)${, so that $(V^{n}_{i})_{1 \leq i \leq k}$ are frozen with probability tending to $1$. In addition,} conditionally on being frozen, $\birth_n(V^n_1)$ is uniform in $\left\{i \in [n]; \ \X^n_i=-1\right\}$. Therefore, by Lemma \ref{lem:hnalpha} (ii), we know that $\birth_n(V^n_1)/n$ converges in distribution towards $U_1$.
The convergence (\ref{eq1 thm sur-critique}) then comes from \ref{hyp:conv1surf}. For (\ref{eq2 thm sur-critique}), we again use \ref{hyp:conv1surf}, using also that $\coal_n(V^n_j,V^n_\ell)/n$ converges in probability towards zero by Lemma \ref{lemme ça coalesce en zéro}.
\end{proof}

\section{Scaling limits: critical regime}
\label{sec:critical}

The goal of Section \ref{sec:critical} is to establish Theorem \ref{thm:critique}. We first explain in Section \ref{sec:deflim} the strategy of the proof of Theorem \ref{thm:critique}.  Roughly speaking, the identification of the ``finite dimensional marginal distributions'' (Gromov--Prokhorov convergence) is based on a continuous-time coalescent process,  defined in Section \ref{coalescent continu}, and which is closely related to a non-homogeneous analogue of Kingman's coalescent introduced by Aldous in \cite{A98} . The law of the genealogy of clusters in this continuous-time coalescent process can be explicitly described (Section \ref{ssec:genealogy}), which allows us to establish the Gromov--Prokhorov convergence of Theorem \ref{thm:critique}. Tightness is then established in  Section \ref{ssec:tight}.

\begin{table}[htbp]\caption{{Table of the main notation and symbols introduced in Section \ref{sec:critical} and used later.}}
\centering
{\begin{tabular}{c c p{12cm} }
\toprule
$\mathfrak{F}^+_{r,k}$ && the set of ordered forests of $r+1$ plane binary trees with $2k-1-r$ vertices with a vertex-labelling from $r+1$ to $2k-1$ which increases along the branches. \\
$\Delta^{n}$ &&
       quantity defined for $u,v \in  \V_{n}(\Xb)$ by $\displaystyle \Delta^{n}(u,v) = \frac{1}{2}(\birth_n(u)+\birth_n(v)-2 \coal_n(u,v))$. \\
\bottomrule
\end{tabular}
}
\label{tab:seccritical}
\end{table}

{
From now on we assume that the assumptions of Theorem \ref{thm:critique} are in force.
}

\subsection{Definition of the limit}
\label{sec:deflim}

The way  we define the limiting compact measured tree $\T(f)$ is rather indirect. For every $ k \geq 1$, let $V_1^n,\ldots,V_k^n$ be $k$ i.i.d.~vertices in $\V_{n}$ chosen uniformly at random (conditionally given $\T^n$). Our strategy is to establish the following two facts:
\begin{enumerate}[noitemsep,nolistsep]
\item[--] Gromov--Prokhorov convergence: 
\begin{equation}
\label{eq:GP}
\left(\frac{d^n(V_j^n,V_\ell^n)}{\sqrt{n}} \right)_{1\le j,\ell\le k}
\end{equation}
converges in distribution as $ n \rightarrow \infty$, where we recall that $d^{n}$ denotes the graph distance on $\T^{n}$. This is established in Sec.~\ref{sec:GPcv}
\item[--] ``Leaf-tightness'': for every $\varepsilon>0$
\begin{equation}
\label{eq:GH}
\lim_{k\to \infty} \limsup_{n\to \infty} \proba{\max_{v\in \V_{n}}  \min_{1\leq i \leq k} \frac{d^n(v,V^{n}_i)}{\sqrt{n}}>\e}=0.
\end{equation}
This is established in Sec.~\ref{ssec:tight}.
\end{enumerate} 

Indeed, by Lemma \ref{C.1}, \eqref{eq:GP} and \eqref{eq:GH} entail the existence of a limiting compact metric space $\T(f)$ equipped with a probability measure $\mu$ with full support such that the convergence of Theorem \ref{thm:critique} holds. The compact metric space $\T(f)$ is a real tree, being the Gromov--Hausdorff limit of real trees.

\subsection{A continuous-time coalescent}
\label{coalescent continu}

Here we introduce a continuous-time coalescent process, which appears in the limit of the quantity \eqref{eq:GP}.  Let $f:[0,1] \to \R_+$ be a measurable function
satisfying the following conditions:
\[ \int_0^1 \frac{1}{f(t)}\mathrm{d}t < \infty, \quad ; \quad \forall a,b \in (0,1), \ \int_a^b \frac{1}{f(t)^2} \mathrm{d}t <\infty. \]
The continuous coalescent construction proceeds as follows (see Figure \ref{fig aldous}): let each of $k$ particles be born at independent uniform random times $B_1,\ldots,B_k$ in $[0,1]$. Particles coalesce into clusters according to the following rule: in time $[t,t-\mathrm{d}t]$, each pair of clusters merges with rate $\frac{1}{f(t)^2}$. If we furthermore assume
\begin{equation}\label{condition3f}
\forall a \in (0,1), \qquad \int_0^a \frac{1}{f(t)^2} \mathrm{d}t = \infty
\end{equation}
this ensures that all the particles will eventually merge into one cluster (Aldous assumes this condition in \cite{A98} but it will not be necessary for our purpose). Since we will not assume this condition, two clusters may not merge with positive probability. If there are $r$ clusters which have not coalesced at time $0$, we  denote by $0<C_{r+1}< \ldots< C_{k-1}$ the times of coalescence into clusters in the increasing order (they are a.s. distinct) and by convention, we also set $C_1= \cdots = C_r=0$.

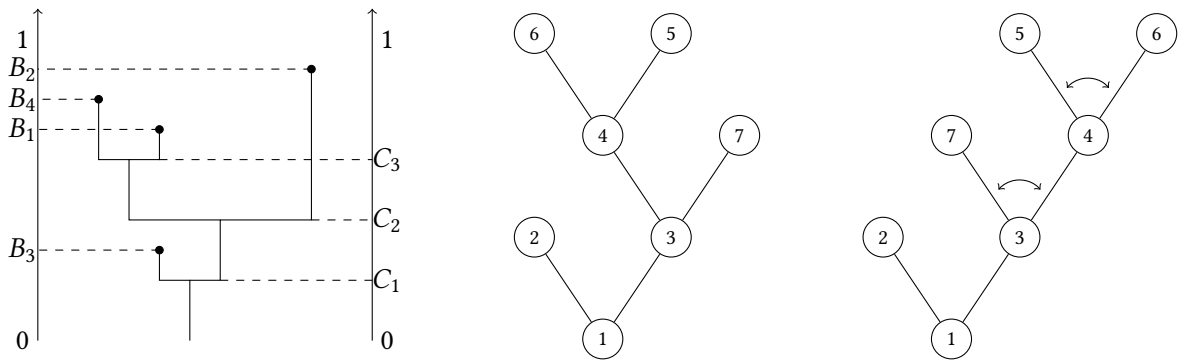
\begin{figure}[h!]
\begin{center}
\begin{tikzpicture}[scale=0.4,
point/.style={draw,circle,fill,inner sep=1pt}]
\draw[->] (0,0) -- (0,11);
\draw[->] (11,0) -- (11,11);

\draw (5,0) -- (5,2) -- (4,2) -- (4,3);
\draw (5,2) -- (6,2) -- (6,4) -- (9,4) -- (9,9);
\draw (6,4) -- (3,4) -- (3,6) -- (4,6) -- (4,7);
\draw (3,6) -- (2,6) -- (2,8);

\node[point] (1) at (4,3) {};
\node[point] (2) at (9,9) {};
\node[point] (3) at (4,7) {};
\node[point] (4) at (2,8) {};

\node (a) at (-0.5,0) {$0$};
\node (b) at (11.5,0) {$0$};
\node (c) at (-0.5,10) {$1$};
\node (d) at (11.5,10) {$1$};

\node (b1) at (-0.5,7) {$B_1$};
\node (b2) at (-0.5,9) {$B_2$};
\node (b3) at (-0.5,3) {$B_3$};
\node (b4) at (-0.5,8) {$B_4$};

\draw[dashed] (4,3) -- (0,3);
\draw[dashed] (9,9) -- (0,9);
\draw[dashed] (4,7) -- (0,7);
\draw[dashed] (2,8) -- (0,8);

\draw[dashed] (6,2) -- (11,2);
\draw[dashed] (9,4) -- (11,4);
\draw[dashed] (4,6) -- (11,6);

\node (c1) at (11.5,2) {$C_1$};
\node (c2) at (11.5,4) {$C_2$};
\node (c3) at (11.5,6) {$C_3$};

\end{tikzpicture}
\hspace{3em}
\begin{tikzpicture}[scale=0.9,
sommet/.style = {draw,circle, font=\scriptsize,inner sep=0,minimum size=15pt}]

\node[sommet] (1) at (0,0) {$1$};
\node[sommet] (2) at (-1,1.5) {$2$};
\node[sommet] (3) at (1,1.5) {$3$};
\node[sommet] (4) at (0,3) {$4$};
\node[sommet] (5) at (1,4.5) {$5$};
\node[sommet] (6) at (-1,4.5) {$6$};
\node[sommet] (7) at (2,3) {$7$};

\draw (1) -- (3) -- (4) -- (6);
\draw (1) -- (2);
\draw (3) -- (7);
\draw (4) -- (5);

\end{tikzpicture}
\hspace{3em}
\begin{tikzpicture}[scale=0.9,
sommet/.style = {draw,circle, font=\scriptsize,inner sep=0,minimum size=15pt}]

\node[sommet] (1) at (0,0) {$1$};
\node[sommet] (2) at (-1,1.5) {$2$};
\node[sommet] (3) at (1,1.5) {$3$};
\node[sommet] (4) at (0,3) {$7$};
\node[sommet] (5) at (3,4.5) {$6$};
\node[sommet] (6) at (1,4.5) {$5$};
\node[sommet] (7) at (2,3) {$4$};

\draw (1) -- (3) -- (4);
\draw (1) -- (2);
\draw (3) -- (7) -- (6);
\draw (7) -- (5);

\draw[<->] (1.3,2.2) arc(40:140:0.4);
\draw[<->] (2.3,3.7) arc(40:140:0.4);
\end{tikzpicture}
\caption{Illustration of the continuous coalescence construction on the left for $k=4$. In the middle is represented the binary forest $\mathcal F$ composed here of only one tree  associated with this construction. On the right is the binary forest where the vertices of each internal node have been exchanged uniformly at random. In this case, the child of the vertices $3$ and $4$ have been exchanged but those of $1$ have not.}
\label{fig aldous}
\end{center}
\end{figure}

This process is actually a time-change of a non-homogeneous analogue of Kingman's coalescent introduced by Aldous in \cite{A98}, see Remark \ref{rem:timechange} below.

\subsection{The genealogy of clusters in the continuous-time coalescent}
\label{ssec:genealogy}

The genealogy of the clusters is described by an increasing binary forest $\mathcal{F}$. The trees forming the forest are associated with the clusters which have not merged at time $0$. Each internal vertex in one of these trees corresponds to a coalescence of two clusters and the leafs correspond to the particles' births. We also give a uniform random order to the children of each internal node, so that the trees of $\mathcal{F}$ can be viewed as plane trees. We equip $\mathcal{F}$ with a labelling from $r+1$ to $2k-1$ of its vertices according to the order of the times of coalescence or birth. If $k>r\ge 0$, we denote by $\mathfrak{F}^+_{r,k}$ the set of ordered forests of $r+1$ plane binary trees with $2k-1-r$ vertices with a vertex-labelling from $r+1$ to $2k-1$ which increases along the branches. If there are $r+1$ clusters at time $0$, then the forest $\mathcal{F}$ belongs to $\mathfrak{F}^+_{r,k}$. 
To express the law of $\left(\mathcal{F}, (C_j)_{j \in [k-1]}, (B_j)_{j \in [k]}\right)$, we first describe what will be its support.

\begin{definition}
Let $k>r\ge 0$. Let $\mathcal{F}_0 \in \mathfrak{F}^+_{r,k}$. Let $b_1,\ldots,b_k \in \R_+$ and $0\le c_{r+1}<\ldots < c_{k-1}$ be distinct real numbers. The triplet $\left(\mathcal{F}_0, (c_j)_{j \in \{r+1,\ldots,k-1\}}, (b_j)_{j \in [k]}\right)$ is called admissible if there is an increasing bijection $\varphi: \left\{c_j ; \ j\in \{ r+1,\ldots, k-1 \} \right\} \cup \left\{b_j ; \ j \in [k]\right\} \to \{r+1,\ldots,2k-1\}$ such that for all $j \in [k]$, the integer $\varphi(b_j)$ is the label of a leaf in $\mathcal{F}_0$ and for all $j \in  \{r+1,\ldots,k-1\}$, the integer $\varphi(c_j)$ is the label of an internal vertex of $\mathcal{F}_0$.
\end{definition}

Roughly speaking, $\left(\mathcal{F}_0, (c_j)_{j \in \{r+1,\ldots,k-1\}}, (b_j)_{j \in [k]}\right)$  is admissible if it is possible to associate the times $(c_{j})$ with cluster coalescence times and $(b_{j})_{j \in [k]}$ with particle birth times  in such a way that their relative order is compatible with $\mathcal{F}_0$. In the example of Fig.~\ref{fig aldous} we have $r=0$, $B_{1}=7/10$, $B_{2}=9/10$, $B_{3}=3/10$, $B_{4}=8/10$, $C_{1}=2/10$, $C_{2}=4/10$, $C_{3}=6/10$ and $\varphi$ is defined by $\varphi(2/10)=1$, $\varphi(3/10)=2$, $\varphi(4/10)=3$, $\varphi(6/10)=4$, $\varphi(7/10)=5$, $\varphi(8/10)=6$, $\varphi(9/10)=7$.

\begin{proposition}\label{loi du coalescent généralisé}
For all $k>r\ge 0$, for all $\mathcal{F}_0 \in \mathfrak{F}^+_{r,k}$, for all $b_1, \ldots , b_k \in [0,1]$ and $0<c_{r+1}<\ldots < c_{k-1}<1$ distinct real numbers such that the triplet $\left(\mathcal{F}_0, (c_j)_{j \in \{r+1,\ldots, k-1\}}, (b_j)_{j \in [k]}\right)$ is admissible, we set
\[
g_{r,k}(\mathcal{F}_0,(b_j)_{1\le j\le k}, (c_j)_{j \in \{r+1,\ldots, k-1\}})=   \frac{1}{2^{k-1-r}}\exp\left( -\int_0^1 \frac{ \binom{a(t)}{2}}{f(t)^2} \mathrm{d}t \right) \prod_{j=r+1}^{k-1} \frac{1}{f(c_j)^2}
\]
where $a(t) \coloneqq \#\{j\in [k] ; b_j>t \} -\#\{j\in \{r+1,\ldots, k-1\} ; c_j>t \}$ corresponds to the number of clusters at time $t$.
Then for all $k\ge 1$, the joint law of the increasing binary forest $\mathcal{F}$ associated with the genealogy (where the order of the children of each inner vertex is chosen uniformly at random), of the coalescing times $C_1\le \ldots\le C_{k-1}$ and birth times $B_1,\ldots,B_k$ is given by: for all $r \in \llbracket 0, k-1 \rrbracket$, for all $\mathcal{F}_0 \in \mathfrak{F}^+_{r,k}$, 
\begin{align*}
\P&(\mathcal{F}=\mathcal{F}_0, B_1 \in \mathrm{d}b_1, \ldots, B_k \in \mathrm{d}b_k, C_1=0,\ldots,C_r=0,C_{r+1} \in \mathrm{d}c_{r+1}, \ldots, C_{k-1} \in \mathrm{d}c_{k-1}) \\
&=g_{r,k}\left(\mathcal{F}_0, (b_j)_{j \in [k]}, (c_j)_{j \in \{r+1,\ldots, k-1\}}\right)\mathrm{d}b_1 \cdots \mathrm{d}b_k \mathrm{d}c_{r+1} \cdots \mathrm{d}c_{k-1}.
\end{align*}
\end{proposition}
\begin{proof}
The proof is the same as the proof of the density formula (2) in \cite{A98} except that we use that the probability that none of the $r+1$ clusters coalesce between time $c_{r+1}$ and $0$ is
\[
\exp \left({-\int_0^{c_{r+1}} \frac{\binom{r+1}{2}}{f(t)^2} \mathrm{d}t} \right).
\]
The factor ${1}/{2^{k-1-r}}$ comes from the uniform order assigned to the children of each internal node of $\mathcal{F}$.
\end{proof}

\begin{remark}
\label{rem:timechange}
The change of time relating the continuous coalescent in \cite{A98} and the continuous coalescent presented here is constructed as follows: for all $s \in [0,1]$, let $G(s) = \int_0^s \frac{1}{2f(t)} \mathrm{d}t$, which is an increasing continuous function from $[0,1]$ to $\R_+$. We define $F$ as the inverse of $G$, extended to $\R_+$ by setting $F(t)=1$ if $t>G(1)$. Let $\ell(t)=2f(F(t))$
. Then one can show that if we change the time by $G$, the continuous coalescent becomes the one introduced in \cite{A98}. Besides, if we assume  condition \eqref{condition3f}, then the law of the coalescent is fully characterized by the density $g_{0,k}$ by Proposition \ref{loi du coalescent généralisé}, which is nothing else but the time-change of the density (2) in \cite{A98}. 
\end{remark}

\subsection{Gromov--Prokhorov convergence}
\label{sec:GPcv}

Our goal is now to establish \eqref{eq:GP}. Recall from \eqref{eq:labels} the notation $\Fbb_{n}(\Xb^n)=  \{i \in [n]: \X_{i}^n=-1\}$, which denotes the labels of the frozen vertices of $\T
^{n}$.
Let $V_1^n,\ldots,V_k^n$ be $k$ vertices in $\T^n$ chosen uniformly (conditionally to the construction of $\T^n$). Let $\mathcal{B}^n$ be the (rooted non-plane) binary tree associated with the genealogy of those $k$ vertices in the coalescent construction: its leaves correspond to the birth-times $B_1^n\coloneqq \birth_n(V_1^n),\ldots, B_k^n\coloneqq \birth_n(V_k^n)$, of $V_1^n,\ldots,V_k^n$ and its internal vertices correspond to a coalescence of two trees which contain each one a $V^n_i$ for $i\in [k]$. 

Note that $B_1^n+1,\ldots, B_k^n+1 \in \Fbb_{n}(\Xb^n)$ with large probability when $n\to \infty$, since the number of non frozen vertices at time $n$ is of order $\sqrt{n}$. Moreover, conditionally on the fact that they belong to $\Fbb_{n}(\Xb^n)$, the integers $B_1^n+1, \ldots, B_k^n+1$ are taken uniformly in $\{ i \in  [n]; \X^n_i=-1\}$. In the sequel, for the sake of simplicity, we will thus assume implicitly that they belong to $\Fbb_{n}(\Xb^n)$.  

Each coalescence between two {rooted} trees both containing vertices among $V_1^n,\ldots,V_k^n$ in the coalescent construction is associated with an internal node $v$ of $\mathcal{B}^n$ and we write those coalescence times in the increasing order $C^n_1<\ldots<C^n_{k-1}$. We also label the vertices of $\mathcal{B}^n$ from $1$ to $2k-1$ in the increasing order of the associated coalescence times or birth times. We then equip $\mathcal{B}^n$ with an ordering of the children of each internal nodes, taking into account which one of the two roots remains the root of the tree after the coalescence, so that $\mathcal{B}^n$ is a plane tree. With these additional structures, $\mathcal{B}^n \in \mathfrak{F}^+_{0,k}$. If $r \in \{0,\ldots,k-1\}$, we denote by $\mathcal{F}^n_r$ the increasing binary forest in $\mathfrak{F}^+_{r,k}$ which describes the genealogy but forgetting the $r$ last coalescences. 

Let us state a local limit estimate for $\left(\mathcal{F}^n, (B_j^n)_{j \in [k]}, (C_j^n)_{j \in [k-1]} \right)$. 

\begin{lemma}
Assume \ref{hyp:conv1surf} and \ref{hyp:conv1surf2}. Let $k>r\ge 0$. Let $\mathcal{F}_0 \in \mathfrak{F}^+_{r,k}$. 
For all $n\ge 1$, independently of $\T^n$, let $\widetilde{B}_1^n,\ldots,\widetilde{B}_k^n$ be i.i.d.\@ uniform random variables in $\left\{0\le i <n , \ \X^n_{i+1}=-1\right\}$ and let $\widetilde{C}_{r+1}^n,\ldots,\widetilde{C}_{k-1}^n$ be an independent sequence of i.i.d.\@ uniform random variables in $\left\{0\le i <n, \ \X^n_{i+1}=1\right\}$.  Let $\widetilde{B}_1, \ldots , \widetilde{B}_k $ and $\widetilde{C}_{r+1},\ldots,  \widetilde{C}_{k-1}$ be i.i.d.\@ uniform random variables in $[0,1]$. Then we have the convergence in distribution
\begin{align}
\left(\frac{n }{2}\right)^{2k-1-r}\P&\left( \left. 
\begin{matrix}
	\mathcal{F}^n_r=\mathcal{F}_0, \  \forall j \in [k],  B^n_j=\widetilde{B}_j^n, \forall j \in [r], C^n_j\le \e n  \\
	\text{ and } \forall j \in \{r+1,\ldots,k-1\},  C^n_j=\widetilde{C}_j^n  
\end{matrix}
\right\vert (\widetilde{B}^n_j)_{j \in [r]}, (\widetilde{C}^n_j)_{r+1 \le j \le k-1} \right) \notag\\
&\cvloi[n]\mathop{\longrightarrow}\limits_{\e \to 0}^{(d)}
 g_{r,k}\left(\mathcal{F}_0, (\widetilde{B}_j)_{j\in [k]}, (\widetilde{C}_j)_{j \in \{r+1,\ldots, k-1 \}}\right),\label{eq:cv locale}
\end{align}
where by convention the function $g_{r,k}$ vanishes on non-admissible triples. In particular, if, for all $i,j\in [k]$ we denote by $C_{i,j}$ the coalescing time between the cluster containing the $i$-th particle with the cluster containing the $j$-th particle, then we have
\begin{equation}
\label{cv coalescent}
\left(\left({B_j^n}/{n}\right)_{j \in [k]}, \left({\coal_n(V^n_i,V^n_j)}/{n}\right)_{i,j \in [k]}\right) \cvloi[n]  \left( (B_j)_{j \in [k]}, (C_{i,j})_{i,j \in [k]}\right) .
\end{equation}
\end{lemma}
\begin{proof}
By Lemma \ref{lem:hnalpha} (ii) and by Skorokhod's representation theorem we may assume that
$$
\forall j \in [k], \ \frac{\widetilde{B}^n_j}{n} \cvps[n] \widetilde{B}_j \qquad \text{and}
\qquad
\forall j \in \{r+1, \ldots k-1\}, \ 
\frac{\widetilde{C}^n_j}{n}
\cvps
\widetilde{C}_j.
$$

For all $i\in[n]$, we set
\[a^n_i=\#\{j\in [k]; \ \widetilde{B}_j^n\ge i\}-\#\{j\in \{r+1,\ldots,k-1\}; \ \widetilde{C}_j^n\ge i\}.\]
Notice that if we replace the $\widetilde{C}^n_j$ by $C^n_j$ and $\widetilde{B}^n_j$ by $B^n_j$, then for $i\ge C_{r}^n+1$ we have that $a^n_i$ is the number of trees in $\F^n_i(\Xb^n)$ containing vertices among $V^n_1,\ldots,V^n_k$.

Let us denote by $\widetilde{\P}$ the conditional probability $\P(\cdot \vert (\widetilde{B}^n_j)_{j \in [k]}, (\widetilde{C}^n_j)_{j \in \{r+1,\ldots, k-1\}})$. Then, {by the definition of Algorithm \ref{algo2}}, when the triple $(\mathcal{F}_0, (\widetilde{B}_j)_{j\in [k]}, (\widetilde{C}_j)_{j \in \{r+1,\ldots, k-1 \}})$ is admissible, we have
\begin{align*}
\widetilde{\P}&\left(C_1^n\le \e n, \ldots, C_r^n \le \e n, \mathcal{F}^n_r=\mathcal{F}_0, \  \forall j \in [k],  B^n_j=\widetilde{B}_j^n \text{ and } \forall j \in \{r+1,\ldots, k-1\},  C^n_j=\widetilde{C}_j^n \right) \\
=&\widetilde{\P}\left( \mathcal{F}^n_r=\mathcal{F}_0, \  \forall j \in [k],  B^n_j=\widetilde{B}_j^n \text{ and } \forall j \in \{r+1,\ldots, k-1\},  C^n_j=\widetilde{C}_j^n \right)\\
&-\widetilde{\P}\left(C_r^n > \e n, \mathcal{F}^n_r=\mathcal{F}_0, \  \forall j \in [k],  B^n_j=\widetilde{B}_j^n \text{ and } \forall j \in \{r+1,\ldots, k-1\},  C^n_j=\widetilde{C}_j^n \right) \\
= &\frac{ 1}{k!{\binom{{(n-S^n_n+1)  / 2 +S_n^n}}{k}}} \left( \prod_{j=r+2}^{k} 
\prod_{\substack{i=\widetilde{C}^n_{j-1}+2\\ \text{s.t. }\X^n_i=1}}^{\widetilde{C}^n_j} \left( 1- \frac{\binom{a^n_i}{2}}{\binom{S_i^n}{2}}\right)\right) 
\prod_{j=r+1}^{k-1} \frac{1}{\displaystyle2\binom{S^n_{\widetilde{C}^n_j+1}}{2}}\\
&-\sum_{\e n <\widetilde{C}^n_r\le n} \binom{r+1}{2}\frac{ 1}{k!{\binom{{(n-S^n_n+1)  / 2 +S_n^n}}{k}}} \left( \prod_{j=r+1}^{k} 
\prod_{\substack{i=\widetilde{C}^n_{j-1}+2\\ \text{s.t. }\X^n_i=1}}^{\widetilde{C}^n_j} \left( 1- \frac{\binom{a^n_i}{2}}{\binom{S_i^n}{2}}\right)\right) 
\prod_{j=r}^{k-1} \frac{1}{2\displaystyle\binom{S^n_{\widetilde{C}^n_j+1}}{2}}
,
\end{align*}
where by convention $\widetilde{C}_k^n= n$. Indeed, the factor $ 1/({k!{\binom{{(n-S^n_n+1) / 2}}{k}}})$ comes from the choice of the uniform (distinct) birth-times $B^n_j$, the factor ${1 / (2\binom{S^n_{\widetilde{C}^n_j+1}}{2})}$ corresponds to the coalescence at time $\widetilde{C}^n_j$ of a pair of trees containing some $V^n_j$ but the choice of the trees is prescribed by $\mathcal{F}_0$ and its planar ordering. The factor $ 1- {\binom{a^n_i}{2} / \binom{S_i^n}{2}}$ encodes the fact that there {is no} coalescence between the trees containing $V^n_1, \ldots, V^n_k$ at times different from $\widetilde{C}^n_{r+1}, \ldots, \widetilde{C}^n_{k-1}$. Finally, the binomial coefficient $\binom{r+1}{2}$ comes from the different possibilities for the coalescence at time $C^n_r$ (since the forest $\mathcal{F}_0$ does not impose the genealogical structure for this coalescence).

By \ref{hyp:conv1surf}, by \ref{hyp:conv1surf2}, by \eqref{eq:cvminSinfini} and using Lemma \ref{lem:hnalpha} (ii),(iv),(v), it is easy to derive the convergence in probability
\begin{align*}
&\left(\frac{n}{2} \right)^{2k-1-r}
\widetilde{\P}\left(C_1^n\le \e n, \ldots, C_r^n \le \e n, \mathcal{F}^n_r=\mathcal{F}_0, \  \forall j \in [k],  B^n_j=\widetilde{B}_j^n \text{ and } \forall j \in \{r+1,\ldots, k-1\},  C^n_j=\widetilde{C}_j^n \right) \\
&\cvproba[n]
\frac{1}{2^{k-1-r}}
 \left(\prod_{j=r+2}^{k}  \exp\left( -\int_{\widetilde{C}_{j-1}}^{\widetilde{C}_j} \frac{ \binom{a(t)}{2}}{f(t)^2} \mathrm{d}t \right) \right)
\left( \prod_{j=r+1}^{k-1}\frac{1}{f(\widetilde{C}_j)^2} \right)
\left(1-\binom{r+1}{2} \int_\e^{\widetilde{C}_{r+1}} \exp\left(-\int_c^{\widetilde{C}_{r+1}} \frac{\binom{r+1}{2}}{f(t)^2}\mathrm{d}t\right) \frac{1}{f(c)^2} \mathrm{d}c\right)\\
&=
\frac{1}{2^{k-1-r}}   \exp\left( -\int_{\widetilde{C}_{r+1}}^{1} \frac{ \binom{a(t)}{2}}{f(t)^2} \mathrm{d}t \right) 
\left( \prod_{j=r+1}^{k-1}\frac{1}{f(\widetilde{C}_j)^2} \right)
\left(1-\binom{r+1}{2} \int_\e^{\widetilde{C}_{r+1}} \exp\left(-\int_c^{\widetilde{C}_{r+1}} \frac{\binom{r+1}{2}}{f(t)^2}\mathrm{d}t\right) \frac{1}{f(c)^2} \mathrm{d}c\right)
,
\end{align*}
where by convention $\widetilde{C}_k=1$, where $a(t)$ is defined in Proposition \ref{loi du coalescent généralisé} and where the convergence in probability of $n/(S^n_{\widetilde{C}^n_j +1})^2$ towards $1/f(\widetilde{C}_j)^2$ comes from the points (ii) and (iv) of Lemma \ref{lem:hnalpha}. Besides,
\[
1-\binom{r+1}{2} \int_\e^{\widetilde{C}_{r+1}} \exp\left(-\int_c^{\widetilde{C}_{r+1}} \frac{\binom{r+1}{2}}{f(t)^2}\mathrm{d}t\right) \frac{1}{f(c)^2} \mathrm{d}c
\mathop{\longrightarrow}\limits_{\e \to 0}\exp\left(- \int_0^{\widetilde{C}_{r+1}} \frac{\binom{r+1}{2}}{f(t)^2} \mathrm{d}t\right),
\]
hence (\ref{eq:cv locale}). 

For all $n\ge 1, \e>0$, let $R^n(\e)$ be the smallest $r'\ge 0$ such that $C^n_{r'}> \e n $. For all $(b^n_j)_{j\in [k]}, (c^n_j)_{j \in \{r+1,\ldots, k-1 \}}$, let
\begin{align*}
g_{\e,n,r,k}&\left(\mathcal{F}_0, (b^n_j)_{j\in [k]}, (c^n_j)_{j \in \{r+1,\ldots, k-1 \}}\right) \\
&=\left(n-\#\mathbb{F}_n({\Xb}^n)\right)^{k-1-r}(\# \mathbb{F}_n(\Xb^n))^{k} \P\left( \mathcal{F}^n_{R^n(\e)}=\mathcal{F}_0, \  \forall j \in [k],  B^n_j=b_j^n \text{ and } \forall j \in \{r+1,\ldots,k-1\},  C^n_j=c_j^n  \right).
\end{align*} 
Then, since $\#\mathbb{F}_n({\Xb}^n) \sim n/2$ as $n \rightarrow \infty$, (\ref{eq:cv locale}) can be rewritten as the convergence in probability
\begin{equation}\label{eq:cv locale bis}
\lim_{\e \to 0}\lim_{n\to \infty} g_{\e,n,r,k}\left(\mathcal{F}_0, (\widetilde{B}^n_j)_{j\in [k]}, (\widetilde{C}^n_j)_{j \in \{r+1,\ldots, k-1 \}}\right) 
=
 g_{r,k}\left(\mathcal{F}_0, (\widetilde{B}_j)_{j\in [k]}, (\widetilde{C}_j)_{j \in \{r+1,\ldots, k-1 \}}\right).
\end{equation}
Let us check that this implies the convergence in law
\begin{equation}\label{eq:cv loi foret et temps}
\left( \mathcal{F}^n_{R^n(\e)}, \  \left(\frac{B^n_j}{n}\right)_{j \in [k]},  \left(\frac{C^n_j}{n}\right)_{j \in [k-1]} \right)
\mathop{\longrightarrow}\limits_{n\to \infty, \e \to 0}^{(d)}
\left( \mathcal{F}, \  \left(B_j\right)_{j \in [k]},  \left(C_j\right)_{j \in [k-1]} \right).
\end{equation}
But by definition of $g_{\e,n,r,k}$ we have
$$
	\P\left( \mathcal{F}^n_{R^n(\e)}=\mathcal{F}_0 \right) =   
	 \E\left[
	g_{\e, n, r, k}\left(\mathcal{F}_0, \left(\frac{\widetilde{B}^n_j}{n}\right)_{j \in [k]} \! , \left(\frac{\widetilde{C}^n_j}{n}\right)_{j \in \llbracket r+1, k-1 \rrbracket}
	\right)
	\right].
$$
In addition,
\[
\E\left[ g_{r,k}\left(\mathcal{F}_0, (\widetilde{B}_j)_{j\in [k]}, (\widetilde{C}_j)_{j \in \{r+1,\ldots, k-1 \}}\right) \right]= \P\left( \mathcal{F}=\mathcal{F}_0 \right).
\]
Therefore, using (\ref{eq:cv locale bis}) with Fatou's lemma and then using the fact that if we sum over all possible $\mathcal{F}_0 \in \bigcup_{0\le r \le k-1} \mathfrak{F}^+_{r,k}$ then we get one,  one can see that
\[
\liminf_{\e \to 0} \liminf_{n \to \infty} \P\left( \mathcal{F}^n_{R^n(\e)}=\mathcal{F}_0 \right) 
= \P\left( \mathcal{F}=\mathcal{F}_0 \right),
\]
i.e.
\[
\liminf_{\e \to 0} \liminf_{n \to \infty} \E \left[g_{\e,n,r,k}\left(\mathcal{F}_0, (\widetilde{B}^n_j)_{j\in [k]}, (\widetilde{C}^n_j)_{j \in \{r+1,\ldots, k-1 \}}\right) \right]
=
\E\left[ g_{r,k}\left(\mathcal{F}_0, (\widetilde{B}_j)_{j\in [k]}, (\widetilde{C}_j)_{j \in \{r+1,\ldots, k-1 \}}\right) \right].
\]
{Here we use the following result: if $I$ is a finite set,  $(X_{n}^{k})_{n \geq 1, k \in I},(X^k)_{k \in I}$ are nonnegative random variables, if $X_{n}^{k} \rightarrow X^{k}$ almost surely as $n \rightarrow \infty$ for every $k \in I$, if for every $n \geq 1$ we have $\sum_{k \in I} \esp{X_{n}^{k}}=\sum_{k \in I} \esp{X^{k}}=1$, then $\esp{X_{n}^{k}} \rightarrow \esp{X^{k}}$ as $n \rightarrow \infty$ for every $k\in I$.} 
Thus, by Scheffé's lemma, the convergence in (\ref{eq:cv locale bis}) holds also in $\Lp^1$.
Let $\e_0>0$. Let $F: \bigcup_{0\le r<k} \mathfrak{F}_{r,k} \times [0,1]^k \times [0,1]^{k-1} \to \R$ be a continuous bounded function which is constant on the elements which have one coordinate in $[0,\e_0]$. Then for all $\e \in (0,\e_0)$,
\begin{align*}
	&\E\left[ F \left(\mathcal{F}^n_{R^n(\e)}, \left(\frac{B^n_j}{n}\right)_{j \in [k]}, \left(\frac{{C}^n_j}{n}\right)_{j \in [k-1]}
	\right)\right] \\
	&=  \! \! \!  \! \! \! \sum_{0\le r <k, \mathcal{F}_0\in \mathfrak{F}_{r,k}} \! \! \! \! \! \! \!
	 \E\left[
	F \left(\mathcal{F}_0, \left(\frac{\widetilde{B}^n_j}{n}\right)_{j \in [k]} \! , \left(\frac{\widetilde{C}^n_j}{n}\right)_{j \in [k-1]}
	\right)
	g_{\e, n, r, k}\left(\mathcal{F}_0, \left(\frac{\widetilde{B}^n_j}{n}\right)_{j \in [k]} \! , \left(\frac{\widetilde{C}^n_j}{n}\right)_{j \in \llbracket r+1, k-1 \rrbracket}
	\right)
	\right] \\
	&\mathop{\longrightarrow} \limits_{n \to \infty} \mathop{\longrightarrow} \limits_{\e \to 0}
	\!  \sum_{0\le r <k, \mathcal{F}_0\in \mathfrak{F}_{r,k}} \! \!	\E\left[
	F \left(\mathcal{F}_0, \left(\widetilde{B}_j \right)_{j \in [k]}  , \left( \widetilde{C}_j \right)_{j \in [k-1]}
	\right)
	g_{r, k}\left(\mathcal{F}_0, \left(\widetilde{B}_j\right)_{j \in [k]} , \left({\widetilde{C}_j}\right)_{j \in \llbracket r+1, k-1 \rrbracket}
	\right)
	\right] \\
	&   \qquad \qquad= \E \left[
	F\left(
	\mathcal{F}, (B_j)_{j \in [k]}, (C_j)_{j\in [k-1]}
	\right)
	\right],
\end{align*}
where by convention for all $j \in [r]$ we have set $\widetilde{C}^n_j=0$ and $\widetilde{C}_j=0$;
{here we use the following generalized version of the dominated convergence theorem: if $X_{n}$ converges almost surely to $X$, if $|X_{n}| \leq Y_{n}$ almost surely, if $Y_{n}$ converges in $\Lp^{1}$ to $Y$, then $X_{n}$ converges in $\Lp^{1}$ to $X$.} 
The convergence  (\ref{eq:cv loi foret et temps}) follows.

Now, one can see that given the forest $\mathcal{F}_{R^n(\e)}^n$ and given the order of the births, one can recover the order of the coalescences occurring at times greater than $\e n$, hence identifying to which $C^n_j$'s the $\coal_n(V^n_i,V^n_\ell)$'s which are greater than $\e n$ are equal. More precisely, for all $r \in \{0, \ldots, k-1\}$, for all $\mathcal{F}_0 \in \mathfrak{F}^+_{r,k}$, for all permutation $\sigma$ of $\{1, \ldots, k\}$ encoding a labelling by $1, \ldots, k$ of the leafs of $\mathcal{F}_0$, we associate the subset $I\subseteq [k]^2$ of couples of labelled leaves which have a common ancestor in $\mathcal{F}_0$ and the unique continuous function $\varphi_{\mathcal{F}_0, \sigma}$ from $\{ (c_{r+1}, \ldots, c_{k-1}) \in [0,1]^{k-r-1}, \ \e<c_{r+1} < \ldots <c_{k-1} \}$ to $[0,1]^I$ which sends $(c_{r+1}, \ldots, c_{k-1})$ to the unique family $(c_{i,j})_{(i,j) \in I}$ such that for all $(i,j) \in I$, the entry $c_{i,j}$ equals $c_\ell$ where the nearest common ancestor of the leaves labelled $i,j$ by $\sigma$ has the $\ell$-th smallest label among the internal nodes of the labelled forest $\mathcal{F}_0$.

Let $\e >0$ and $n\ge 1$. Let $\pi^n$ (resp. $\pi$) be the random uniform permutation associated with the ordering of $B^n_1, \ldots, B^n_k$ (resp. $B_1, \ldots, B_k$), which defines a labelling of the leaves of $\mathcal{F}^n_{R^n(\e)}$ (resp. $\mathcal{F}$).  Let $I^n_\e = \{ (i,j) \in [k]^2, \ \coal_n(V^n_i,V^n_j) > \e n\}$. Then
\[
\left(\frac{\coal_n(V^n_i,V^n_j)}{n}\right)_{(i,j) \in I^n_\e}
 = 
\varphi_{\mathcal{F}^n_{R^n(\e)},\pi^n} \left(\left(\frac{C^n_j}{n}\right)_{j \in \{R^n(\e)+1, \ldots, k-1\}}\right),
\]
hence the convergence in law (\ref{cv coalescent}) by (\ref{eq:cv loi foret et temps}).
\end{proof}

\begin{proof}[Proof of  \eqref{eq:GP}.]
Now, let us prove the convergence in law of $(d^n(V_j^n,V_\ell^n)/\sqrt{n})_{1\le j,\ell \le k}$. 
For all $j,\ell \in [k]$,
\[
\dc^n(V_j^n,V_\ell^n) = \sum_{\substack{i=\coal_n(V^n_j, V^n_\ell)}}^{B^n_j} \frac{1}{2S^n_i}
+
\sum_{\substack{i=\coal_n(V^n_j, V^n_\ell)}}^{B^n_\ell} \frac{1}{2S^n_i}.
\]
Using \ref{hyp:conv1surf}, Lemma \ref{lem:hnalpha} (ii),(iv), {Proposition \ref{height=ok}} and the convergence (\ref{cv coalescent}), it is henceforth straightforward to conclude that
\[
\label{eq:cvdn}
\left(\frac{d^n(V_j^n,V_\ell^n)}{\sqrt{n}} \right)_{1\le j,\ell\le k}
\cvloi[n]
\left(\int_{C_{j,\ell}}^{B_j} \frac{1}{2 f(t)} \mathrm{d}t
+\int_{C_{j,\ell}}^{B_\ell} \frac{1}{2 f(t)} \mathrm{d}t
\right)_{1\le j,\ell \le k}.
\]
\end{proof}

\subsection{Gromov--Hausdorff tightness}
\label{ssec:tight}

We now establish \eqref{eq:GH}. We first  introduce some notation: for $n \geq 1$, let $V^n$ and $(V_i^n)_{ i \geq 1	}$ be independent uniform vertices in $\V_{n}$. For simplicity, we drop the superscript and write $V, (V_{i})_{i \geq 1}$. For $u,v\in \V_{n}$ and $B \subset \V_{n}$ we write
\[
\Delta^{n}(u,v) \coloneqq \frac{1}{2}(\birth_n(u)+\birth_n(v)-2 \coal_n(u,v)), \qquad \Delta^n(u,B) \coloneqq \min_{v\in B} \Delta^{n}(u,v).
\]
{Observe that $\Delta^{n}$ satisfies the triangle inequality.}
Finally, we set
\[M_{n}\coloneqq\max_{0\leq i \leq n} S^n_i.\]
Our main technical input is the following uniform estimate.

\begin{proposition}
\label{tightness time estimate}
For every $k,n \geq 1$ large enough with $k\leq n^{1/3}$,
\[ \proba{\Delta^n(V,\{V_1,\dots, V_{k}\}) > \frac{\ln(k)^2}{\sqrt{k}} n } \leq  \frac{1}{k^{4}}.\]
\end{proposition}

The idea underlining the proof of Proposition \ref{tightness time estimate} is the following: on every time-interval of length $n/\sqrt{k}$, we can find roughly $ \sqrt{k}$ vertices among $V_{1}, \ldots, V_{k}$ and the probability that $\sqrt{k}$ vertices do not coalesce with a given vertex during $n/\sqrt{k}$ time steps is  strictly positive. Let us first explain how tightness follows from this estimate.

\begin{proof}[Proof of Theorem \ref{thm:critique} (tightness) using Proposition \ref{tightness time estimate}]
\emph{Step 1.} We first check that for every $\e>0$ we have
\begin{equation}
\label{eq:stepone}\lim_{k\to \infty} \limsup_{n\to \infty} \proba{\max_{v\in \V_{n}} \Delta^n(v,\{V_1,\dots, V_k\})/n>\e}=0.
\end{equation}
To this end, we apply the so-called chaining method. Specifically,  for $k, n \geq 1$ we set
\[ \e^n_k \coloneqq   \max_{2^k<i\leq 2^{k+1}} \Delta^n(V_i,\{V_1,\dots, V_{2^k}\}),
\]
and roughly speaking, we show that $ \e^n_k/n\lesssim 2^{-k/2}$, where $\lesssim$ denotes an informal upper bound. Indeed, this entails
\[  \frac{1}{n}\max_{v\in \V_{n}} \Delta^n(v,\{V_1,\dots, V_{2^k}\}) \lesssim  \frac{1}{n}\sum_{i\geq k} \e^n_{i} \lesssim 2^{-k/2}\]
and \eqref{eq:stepone} will follow. Let us now be more precise. For $n \geq 1$, let $N_n$ be the largest integer with $2^{N_n}\leq n^{1/3}$  and set $\e^n_{\infty}\coloneqq\max_{v\in \V_{n}} \Delta^n(v,\{V_1,\dots, V_{2^{N_n}}\})$. By Proposition \ref{tightness time estimate}, for every $k,n$ large enough with $k< N_n$, 
\[ \proba{\e^n_k >k^2 2^{-k/2}n} \leq 2^k\proba{ \Delta^n(V,\{V_1,\dots, V_{2^k}\}) >k^2 2^{-k/2} n }\leq 2^{-3k}.\]
Similarly, since there are at most $n$ vertices in $\T^n$,
\[ \proba{\e^n_\infty >(N_n)^2 2^{-N_n/2}n} \leq n \proba{ \Delta^n(V,\{V_1,\dots, V_{2^{N_n}}\}) >(N_n)^2 2^{-N_n/2}n  }\leq n 2^{-4N_n}=o(1).\]
Hence, for every $k$ large enough, for every $n$ large enough, with probability at least {
\sout{$1-2\sum_{i=k}^\infty 2^{-3i}\geq 1-2^{-k}$} $1-\sum_{i=k}^\infty 2^{-3i}+o(1) \geq 1-2^{-k}+o(1)$, using the triangle inequality,} 
\[ \max_{v\in \V_{n}} \Delta^n(v,\{V_1,\dots, V_{2^k}\}) \leq \sum_{i=k}^{N_n-1} \e^n_i +\e^n_\infty \leq \sum_{i=k}^{N_n} i^2 2^{-i/2}n\leq c 2^{-k/3} n\]
{for some constant $c>0$.}
The convergence \eqref{eq:stepone} readily follows.

\emph{Step 2.} We now show \eqref{eq:GH}  using \eqref{eq:stepone}. By Proposition \ref{height=ok} (applied with $\alpha=1/2$), it suffices to show that for every $\e>0$,
\begin{equation}
\lim_{k\to \infty} \limsup_{n\to \infty} \proba{\max_{v\in \V_{n}} \min_{1\leq i \leq k} \dc^n(v,V_i)/\sqrt{n}>\e}=0. \label{16/01/11h}
\end{equation}
{Now, observe that we may fix $\delta>0$ such that for $n$ sufficiently large:
\begin{equation}
\label{eq:delta}\sum_{i=1}^n  \frac{1}{S^n_{i}} \mathbbm{1}_{S^n_{i}\leq \delta \sqrt{n}} \leq \e \sqrt{n}.
\end{equation}
Indeed, we have
\[
\frac{1}{\sqrt{n}} \sum_{i=1}^n  \frac{1}{S^n_{i}} \mathbbm{1}_{S^n_{i}\leq \delta \sqrt{n}}=\int_{0}^{1}  \frac{\sqrt{n}}{S^{n}_{nt}} \mathbbm{1}_{1/\delta \leq   \frac{\sqrt{n}}{S^{n}_{nt}}} \mathrm{d}t
\]
and \eqref{eq:delta} follows from  \ref{hyp:conv1surf} by using the fact that convergence in $\mathrm{L}^{1}$ implies uniform integrability.}
 
Then, by the definition of $\dc^{n}$, write for every $\delta>0$ and $u,v \in \V_{n}$:
\[  \frac{\dc^n(u,v)}{\sqrt{n}} \leq \sum_{i=1}^n  \frac{1}{\sqrt{n} S^n_{i}} \mathbbm{1}_{S^n_{i}\leq \delta \sqrt{n}}+ \frac{\Delta^{n}(u,v)}{n \delta} \leq \varepsilon+ \frac{\Delta^{n}(u,v)}{n \delta} .\]
The estimate \eqref{16/01/11h} then follows from \eqref{eq:stepone}.
\end{proof}

We now turn to the proof of Proposition \ref{tightness time estimate}. {Recall that for $n\in \N$, $M_{n}\coloneqq\max_{0\leq i \leq n} S^n_i$.} Let $M>0$ be such that $M_{n} \leq M \sqrt{n}$ for all $n \geq 1$ {(this is possible by assumption)}. {For $k,n\in \N$ let $\delta_{k,n}\coloneqq 4 \lceil Mn/\sqrt{k} \rceil$.} Observe that $\delta_{k,n}>M_{n}$ for $k,n$ sufficiently large with $k \leq n^{1/3}$.

\begin{lemma} \label{OneStepTightness}
For every $k,n \geq 1$ large enough with $k \leq n^{1/3}$, for every {$ 0 \leq i \leq n/ \delta_{k,n}$}, for every  $v_{0} \in \V_{n}$ with {$\birth_{n}(v_{0}) > (i+1) \delta_{k,n}$}, for every $K \geq 1$ and { $v_{1}\neq \ldots \neq v_{K} \in \V_{n}$ such that  $i \delta_{k,n} < \birth_{n}(v_{j}) \leq (i+1) \delta_{k,n}  $}, we have
\[
\proba{\textrm{the tree of } \F^{n}_{ (i-1)\delta_{k,n} } \textrm{ containing } v_0 \textrm{ does not contain any vertex among } v_{1}, \ldots,v_{K}} \leq \frac{3 M \sqrt{k}}{4 K}.
\]
\end{lemma}

\begin{proof}
For every $0\leq t \leq n$, we define $N_t^n$ by induction as follows: $N_n^n=K+1$, and then for $1 \leq t \leq n$ we set $N_{t-1}^n=N_t^n-1$ if $\X^n_t=1$ and two trees of $\F^n_t$ which contain a vertex in $\{v_{0},\dots, v_{K}\}$ coalesce at time $t-1$ and  $N_{t-1}^n=N_t^n$ otherwise. In words, $N_t^n$ represents the number of trees in $\F^n_t$ which contain a vertex in $ \{v_{0}, \ldots,v_{K}\}$, taking also into account vertices which are not born yet.

\emph{Step 1.} We first show that
\begin{eqnarray}
\label{eq:step1Delta}
&&\proba{\textrm{the tree of } \F^{n}_{ (i-1)\delta_{k,n} } \textrm{ containing } v_0 \textrm{ does not contain any vertex among } v_{1}, \ldots,v_{K}} \notag \\
& & \qquad \qquad \qquad \qquad \qquad \qquad \qquad\qquad\qquad\qquad\qquad\qquad  \leq \frac{1}{K+1} \esp{N^{n}_{(i-1) \delta_{k,n}  }}.
\end{eqnarray}
To this end, we say that $u\in \{v_{0},\ldots,v_{k}\}$ stays alone at time $t$ if either $t>b_{n}(u)$, or if the tree of $\F^n_t$ which contains $u$ does not contain any other vertex  of $\{v_{0},\dots,v_{k}\} \backslash \{u\} $. Plainly, by definition $N^n_t$ is at least equal to the number of vertices which stay alone at time $t$. Moreover, since $\birth_{n}(v_{0})>\birth_{n}(v_i)$ for $1 \leq i \leq K$,  $v_{0}$ is the vertex with the less chance to stay alone at time $t$ among $ \{v_{0}, \ldots,v_{K}\}$. Hence
\[ \frac{1}{K+1}\esp{N^n_t}\geq \frac{1}{K+1}\sum_{i=0}^{K+1} \proba{v_{i} \text{ stays alone at time } t}\geq  \proba{v_0 \text{ stays alone at time }t}.\]
The bound \eqref{eq:step1Delta} follows by noting that if $v_{0}$ stays alone at time $t$ then the tree of $\F^{n}_{t}$ containing $v_{0}$ does not contain any vertex among $v_{1}, \ldots,v_{K}$, with $t= (i-1) \delta_{k,n} $.

\emph{Step 2.} We next show that
\begin{equation}
\label{eq:step2Delta}  \esp{N^{n}_{ (i-1) \delta_{k,n}  }} \leq 1+\frac{2(M_{n})^2}{\delta_{k,n}-M_{n}-1}
\end{equation}
for $k,n \geq 1$ sufficiently large with $k \leq n^{1/3}$.
For every $ (i-1) \delta_{k,n}   {<} t {\leq} i \delta_{k,n} $, depending on whether two trees which contain a vertex in $\{v_0,\dots, v_{k}\}$ coalesce or not at time $t-1$ in Algorithm \ref{algo2}, we get
\[
 \E[N_{t-1}^n]  = \E[N_{t}^n]-\1_{\Xb^n_t=1} \frac{\E[N_{t}^n(N_{t}^n-1)]}{S^n_t(S^n_t-1)}  \leq \E[N_{t}^n]-\1_{\Xb^n_t=1} \frac{\E[N_{t}^n]\E[ N_{t}^n-1]}{(M_{n})^2}.
\]
Then, for every  $ (i-1) \delta_{k,n}    {<}  t  {\leq}  i \delta_{k,n} $, write
\[
\frac{1}{\E[N_{t-1}^n]-1}-\frac{1}{\E[N_{t}^n]-1} = \frac{\E[N_{t}^n]-\E[N_{t-1}^n]}{(\E[N_{t-1}^n]-1)(\E[N_{t}^n]-1)} \geq  
\frac{\1_{\Xb^n_t=1}\E[N_{t}^n]\E[ N_{t}^n-1]/(M_{n})^2}{(\E[N_{t-1}^n]-1)(\E[N_{t}^n]-1)} \geq \frac{\1_{\Xb^n_t=1}}{(M_{n})^2}.
\]
As a consequence, for $k,n \geq 1$,
\[
\frac{1}{\E\left [N_{(i-1) \delta_{k,n} }^n \right ]-1} \geq  \frac{1}{\E\left [N^n_{ i \delta_{k,n} } \right ]-1 }+\sum_{j=(i-1) \delta_{k,n}+1}^{ i \delta_{k,n} } \frac{\1_{\Xb^{n}_j=1}}{(M_{n})^2} \geq  \frac{S^{n}_{ i \delta_{k,n} }-S^{n}_{ (i-1 )\delta_{k,n} }+ i \delta_{k,n} - (i-1) \delta_{k,n} }{2(M_{n})^2}
\]
and \eqref{eq:step2Delta} follows $k,n \geq 1$ sufficiently large with $k \leq n^{1/3}$ (so that $\delta_{k,n}> M_{n}$).

\emph{Step 3.}  To complete the proof of Lemma \ref{OneStepTightness}, it just remains to note that for $k,n \geq 1$ sufficiently large with  $k\leq n^{1/3}$ we have $\delta_{k,n}-M_n-1\geq 4Mn/\sqrt{k}-M_n-1\geq 3 Mn/\sqrt{k}$ so that
\[
 \frac{1}{K+1} \left( 1+\frac{2(M_{n})^2}{\delta_{k,n}-M_{n}-1} \right) \leq  \frac{1}{K+1} \left( 1+\frac{2 M^{2} n}{3Mn/\sqrt{k}} \right)  \leq \frac{M\sqrt{k}}{K+1} \left( \frac{1}{M\sqrt{k}}+ \frac{2}{3} \right)\leq \frac{3 M\sqrt{k}}{4K}.
\]
This completes the proof.
\end{proof}

\begin{proof}[Proof of Proposition \ref{tightness time estimate}]
{Recall that  $ M_{n} \leq  M \sqrt{n}$ for every $n \geq 1$ and that  $\delta_{k,n}= 4 \lceil  M n/\sqrt{k} \rceil$.}

 To simplify notation, set $\birth_n(V_j)=\birth_{n}^{j}$ for $1 \leq j \leq n$ and $\birth_{n}=\birth_{n}(V)$ and define the event $E$ by
{ \[
E= \left\{\forall 0\leq i \leq \lfloor n/\delta_{k,n}\rfloor-1, \#  \{\birth_{n}^{j} : 1\leq j \leq k \textrm{ and } i\delta_{k,n} < \birth_{n}^{j}  \leq (i+1)\delta_{k,n} \} >  M \sqrt{k}\right\}.
\]}

\emph{Step 1.} We first check that
\begin{equation}
\label{eq:check1}
\proba{E^{c}}= o \left(  \frac{1}{ k^{4}}\right).
\end{equation}
To this end, first note that for every $0\leq i \leq \lfloor n/\delta_{k,n}\rfloor-1$,
\[ \sum_{i\delta_{k,n} <t\leq (i+1) \delta_{k,n}} \1_{\Xb^{n}_t=-1}= (\delta_{k,n}-S^{n}_{(i+1)\delta_{k,n}}+S^{n}_{i\delta_{k,n}})/2  \geq \delta_{k,n}/2-M \sqrt{n}/2 \geq \frac{3 M n}{2 \sqrt{k}}.  \]
for $k,n$ sufficiently large with $k \leq n^{1/3}$.
Since there are at most $n$ vertices in $\T^n$, if follows that we have
$\proba{i\delta_{k,n} < \birth_{n}  \leq (i+1) \delta_{k,n}} \geq {3M}/({2\sqrt{k}})-O(1/\sqrt{n})$.
Next,  if $\mathsf{Bin}(k,p)$ denotes a Binomial random variable with parameter $(k,p)$, by Bennett's inequality (Proposition \ref{prop:bennett}) we have for $0 <u < pk$:
\[ \proba{\mathsf{Bin}(k,p)<u } \leq \exp \left( - pk g \left( 1- \frac{u}{pk} \right) \right).\]
Since $\mathsf{Bin}(k,p)$ is stochastically increasing in $p$, it follows that
\[
 \proba{\#\{ 1 \leq j \leq k: \birth_{n}^{j} \in (i\delta_{k,n} ,(i+1)\delta_{k,n}]\} \leq {4}M \sqrt{k} }\leq \exp \left( -   {  {6} M \sqrt{k}} g(1/3) \right) \leq \exp \left( - \frac{\sqrt{k}}{100}\right),
 \]
 since $M$ can be chosen large.
So, by a union bound,
\[\proba{ \exists 0\leq i \leq \lfloor n/\delta_{k,n}\rfloor-1,\#\{ 1 \leq j \leq k: \birth_{n}^{j} \in (i\delta_{k,n} ,(i+1)\delta_{k,n}]\} \leq   {4} M \sqrt{k} }   \leq
 \frac{n}{\delta_{k,n}}  e^{-\sqrt{k}/100} \leq  \frac{\sqrt{k}}{4M}e^{-\sqrt{k}/100}, \]
 which is $o(k^{-4})$. To establish \eqref{eq:check1} it thus remains to check that
 \[
 {\proba{ \exists 1 \leq  i_{1} < i_{2} < i_{3} <  i_{4} \leq k : \birth_{n}^{i_{1}}= \birth_{n}^{i_{2}}=\birth_{n}^{i_{3}} =\birth_{n}^{i_{4}}}= o \left(  \frac{1}{k^{4}}\right).}
 \]  
 To this end, using the fact that there are at least $n/3$ vertices in $\T^{n}$ for $n$ sufficiently large, a union bound yields {$\proba{ \exists 1 \leq  i_{1} < i_{2} < i_{3} <  i_{4} \leq k : \birth_{n}^{i_{1}}= \birth_{n}^{i_{2}}=\birth_{n}^{i_{3}} =\birth_{n}^{i_{4}}} \leq  27 k^{4}/n^{3} =o(1/k ^{4})$ since $k \leq n^{1/3}$}.

\emph{Step 2.} Next we need some more notation. Set
\[
\mathbb{V}_{n}(i)= \{V_{j} : 1 \leq j \leq k \textrm{ and } i \delta_{k,n} {<} \birth_{n}^{j} {\leq} (i+1) \delta_{k,n} \}.
\]
Observe that under $E$ we have $\#\mathbb{V}_{n}(i) \geq M \sqrt{k}$ for every $0 \leq i \leq \lfloor n/ \delta_{k,n} \rfloor -1$.   Setting $p_{k}={\lfloor 100 \ln(k)\rfloor}$, we show that
\begin{equation}
\label{eq:tocheck}\proba{\Delta^n(V,\{V_1,\dots, V_{k}\})> 2p_{k} \delta_{k,n} \, \big |  \, E} = o(k^{-4}),
\end{equation}
and the desired result will follow

\emph{First case: $\birth_{n}$ is small.} First, if $\birth_{n} \leq 2p_{k}  \delta_{k,n}$, under $E$ there exists $1 \leq j \leq k$ such that $\birth^{j}_{n} \leq 2 \delta_{k,n}$, and then $\Delta^n(V,\{V_1,\dots, V_{k}\}) \leq \Delta^{n}(V,V_{j}) \leq 2 p_{k}\delta_{k,n}$. Thus
\[
\proba{\Delta^n(V,\{V_1,\dots, V_{k}\})> 2p_{k} \delta_{k,n} \, \big |  \, E, \birth_{n} \leq 2 p_{k} \delta_{k,n}} = 0
\]

\emph{Second case. $\birth_{n}$ is large.}  
For $i \geq 1$, let $A_{i}$ be the event:
\begin{center}
``the tree of $\F^n _{ {(i-1)} \delta_{k,n} }$ containing $V$ does not contain any vertex of $\mathbb{V}_{n}(i)$.''
\end{center}
By definition of Algorithm \ref{algo2}, for $i_{0} \geq 2 p_{k}$, conditionally given  $E$ and $i_{0} \delta _{k,n} < \birth_{n} \leq (i_{0}+1) \delta_{k,n}$, the events $(A_{i})_{2 \leq i \leq i_{0}-1, {i \text{ even}}}$ are independent and by Lemma \ref{OneStepTightness} (applied with $K= \lceil M \sqrt{k} \rceil$) their probability is at most $3/4$.  As a consequence, for every $i_{0} \geq 2 p_{k}$:
\[
\proba{ \bigcap_{i=i_{0}-p_{k}-1}^{i_{0}-1} A_{i} \, \middle | \,  E, i_{0} \delta _{k,n} < \birth_{n}  \leq (i_{0}+1) \delta_{k,n}} \leq  (3/4)^{{(p_{k}-1)/2}}.
\]
But then observe that when $i_{0} \delta _{k,n} < \birth_{n} \leq (i_{0}+1) \delta_{k,n}$, if for $i_{0}-p_{k}-1 \leq i \leq i_{0}-1$,  if the tree of $\F^n _{ i \delta_{k,n} }$ containing $V$  contains a vertex of $\mathbb{V}_{n}(i)$, say $V_{j}$, then
\[
\Delta^{n}(V,V_{j}) \leq  \frac{1}{2}  \left( (i_{0}+1) \delta_{k,n}+(i+1) \delta_{k,n}- 2 (i-1) \delta _{k,n} \right) \leq  \left(  \frac{i_{0}-i}{2}  +3 \right)\delta _{k,n}+1 \leq  2 p_{k} \delta_{k,n}
\]
for $k,n$ sufficiently large. We conclude from the previous discussion that
\[
\proba{\Delta^n(V,\{V_1,\dots, V_{k}\})> 2p_{k} \delta_{k,n} \, \middle |  \, E, \birth_{n} > 2 p_{k} \delta_{k,n}} = o  \left( (3/4)^{{(p_{k}-1)/2}} \right)=o(k^{-4}),
\]
where the last equality comes from our choice of $p_{k}$. This completes the proof.
\end{proof}

\subsection{Properties of $\mathcal{T}(f)$}

Here we study some properties of $\mathcal{T}(f)$ by establishing Theorem \ref{thm:properties}. 

\begin{proof}[Proof of Theorem \ref{thm:properties} (1)]
The fact that the mass measure on $\mathcal{T}(f)$ has full support was already mentioned in Section \ref{sec:deflim} as a consequence of Lemma \ref{C.1}, taking into account \eqref{eq:GP} and \eqref{eq:GH}.  The fact that it gives full support to the leaves follows from the property the coalescence time of two particles of the continuous-time coalescent is strictly smaller than the birth times, combined with \eqref{cv coalescent} and \eqref{eq:cvdn}: this implies that for two vertices sampled uniformly at random in $\mathcal{T}(f)$, one of them is not a descendant of the other.
\end{proof}

\begin{proof}[Proof of Theorem \ref{thm:properties} (3)]
Let us first assume that $\int_0 \frac{1}{f^2}<\infty$. To show that the degree of the root of $\T(f)$ is a.s. infinite, it suffices to show that a.s. there exist infinitely many clusters which do not coalesce with each other.
First, observe that since $\int_0 \frac{1}{f^2} <\infty$, we can build a decreasing sequence $(\e_k)_{k\ge 0}$ of positive numbers such that
\[
\int_{\e_{k+1}}^{\e_k} \frac{1}{f^2} \le \frac{2^{-k}}{\binom{k}{2}}.
\]
Next, since $(B_i)_{i\ge 1}$ are i.i.d. uniform random variables, we know that a.s. there exists a subsequence $(B_{i_k})_{k\ge 1}$ such that for all $k\ge 1$, we have $B_{i_k} \in [\e_{k+1},\e_k]$. But then, conditionally on the $B_{i_k}$'s, for all $k_0\ge 1$, the probability that none of the clusters containing a $B_{i_k}$ for $k\ge k_0$ coalesce with each other is
\[
\prod_{k\ge k_0} \exp \left(
-\int_{B_{i_{k+1}}}^{B_{i_k}} \frac{\binom{k}{2}}{f(t)^2} \mathrm{d}t
\right)\ge \prod_{k\ge k_0} \exp \left(
-\int_{\e_{k+1}}^{\e_k} \frac{\binom{k}{2}}{f(t)^2} \mathrm{d}t
\right)
\ge \exp \left(
-\sum_{k\ge k_0} 2^{-k}
\right)
\mathop{\longrightarrow}\limits_{k_0 \to \infty}
1.
\]
Thus, with probability one there exist infinitely many clusters which do not coalesce with each other. Conversely, when $\int_0 \frac{1}{f^2} = \infty$, the probability that two clusters which have not merged yet at time $t>0$ do not coalesce until time zero is 
$
\exp \left(-\int_0^t \frac{1}{f(t)^2} \mathrm{d}t\right) =0,
$ which proves that the root has a.s. degree $1$. For the other branchpoints, the fact that they have degree two is due to the fact that a.s. the coalescences happen at distinct times and thus at distinct heights.
\end{proof}

\subsection{Application: the Brownian CRT}
\label{ssec:CRT}

Here we establish Corollary \ref{cor:CRT}.

\begin{proof}[Proof of Corollary \ref{cor:CRT}]
The idea is to apply Theorem \ref{thm:critique} with a random sequence $(\Xb^{2n+1})_{n \geq 1}$ chosen in such a way that  {$(S^{2n+1}_{k})_{0 \leq k \leq 2n+1}$} has the law of a simple random walk started at $1$ and conditioned to hit $0$ for the first time at time $2n+1$.

First, by \cite[Theorem 2]{BBKK23+}, as a random compact metric space, $\T_{2n+1}$ has the same law as a uniform plane tree with $n+1$ vertices, so that $\frac{1}{\sqrt{2n+1}} \cdot \T_{2n+1}$ converges in distribution to the Brownian CRT \cite{Ald93}.

{Second, we shall show that  $\frac{1}{\sqrt{2n+1}} \cdot \T_{2n+1}$ converges in distribution to $ \mathcal{T}(\mathbbm{e})$, which will imply that $ \mathcal{T}(\mathbbm{e})$ has the law of the Brownian CRT. To this end,  it is standard (see e.g.~\cite{Kai76}) that the convergence
\begin{equation}
\label{eq:limexc} \left( \frac{S^{2n+1}_{(2n+1)t} }{\sqrt{2n+1}}  : 0 \leq t \leq 1 \right) \cvloi \mathbbm{e}
\end{equation}
holds in distribution for the $J_1$-Skorokhod topology and also for the uniform convergence since the limit is continuous, where $\mathbbm{e}$ is the Brownian excursion.}

{ By Skorokhod's representation theorem, we may assume that the convergence \eqref{eq:limexc} holds almost surely. We shall check that the conditions of application of Theorem \ref{thm:critique} are satisfied almost surely with the sequence $(\Xb^{2n+1})_{n \geq 1}$ and $f=\mathbbm{e}$.}

{
Notice that the conditions \ref{hyp:conv1surf2} and $\max_{1 \leq  k \leq 2n+1} S^{2n+1}_{k}= \mathcal{O}(\sqrt{n})$ are satisfied thanks to \eqref{eq:limexc} and to the continuity of $f$. We claim that \ref{hyp:conv1surf} follows if we check that
\begin{equation}
\label{eq:alphatight}\lim_{\delta\to 0} \limsup_{n\to \infty} \frac{1}{\sqrt{n}} \sum_{i=1}^n  \frac{1}{S^{2n+1}_{i}} \mathbbm{1}_{S^{2n+1}_{i}\leq \delta \sqrt{n}}=0.
\end{equation}
Indeed, we have
\[
\frac{1}{\sqrt{n}} \sum_{i=1}^n  \frac{1}{S^{2n+1}_{i}} \mathbbm{1}_{S^{2n+1}_{i}\leq \delta \sqrt{n}}=\int_{0}^{1}  \frac{\sqrt{n}}{S^{2n+1}_{nt}} \mathbbm{1}_{1/\delta \leq   \frac{\sqrt{n}}{S^{2n+1}_{nt}}} \mathrm{d}t,
\]
so \eqref{eq:alphatight} is equivalent to the uniform integrability of $ (\sqrt{2n+1}/S^{2n+1}_{ (2n+1)t })_{0 \leq t \leq 1}$ . Combined with \eqref{eq:limexc}, this implies convergence in $\mathrm{L}^{1}$ and thus   \ref{hyp:conv1surf}.}

It remains to check \eqref{eq:alphatight}.   To simplify notation, for $n \geq 1$ and $0 \leq \eta < \delta \leq 1$ set
\[
X_{n}(\delta)= \frac{1}{\sqrt{n}} \sum_{i=1}^n  \frac{1}{S^{2n+1}_{i}} \mathbbm{1}_{S^{2n+1}_{i}\leq \delta \sqrt{n}}, \qquad X(\eta,\delta)= \int_{0}^{1} \frac{1}{\mathbbm{e}(t)} \mathbbm{1}_{\eta \leq \mathbbm{e}(t) \leq \delta}\mathrm{d} t.
\]
It suffices to check that for every $\e>0$, for a well chosen coupling, almost surely  there exists $\delta>0$ such that  every $n \geq 1$ sufficiently large we have
$X_{n}(\delta) \leq \e$.

\emph{Step 1.} We show that
\begin{equation}
\label{eq:borne} \proba{X_{n}(\delta) \geq \e} \leq  \frac{C}{\e} \delta^{2}
\end{equation}
for all $n\ge 1$ {for a certain constant $C>0$.}
This is essentially \cite[Lemma 3.6]{CMY00}, but let us give a  different proof based on the local limit theorem which gives a slightly better upper bound.
To this end, denote by $(S_{i})_{i \geq 0}$ a simple random walk started from ${1}$.  By Kemperman's formula (see e.g.~\cite[Sec.~6.1]{Pit06}), we have {for $1 \leq i \leq 2n+1$}
\begin{equation}
\label{eq:probalocale}
\proba{S^{2n+1}_{i}=k}= \frac{2n+1}{\proba{S_{2n+1}=0}} \cdot \frac{k-1}{i} \proba{S_{i}=k} \cdot  \frac{k}{2n-i+1} \proba{S_{2n-i+1}=-k+1}.
\end{equation}
By the local limit theorem (see e.g.~\cite[Theorem 4.2.1]{IL71}), we get for all $n \ge 1$, for $1 \leq i \leq n$,
\[
\proba{S^{2n+1}_{i}=k} \leq C n^{3/2} \cdot  \frac{k}{i} \cdot \frac{1}{\sqrt{i}} \cdot  \frac{k}{n} \cdot \frac{1}{\sqrt{n}}= C  \frac{k^{2}}{i^{3/2}},
\]
where $C$ is a constant that may change from line to line. It follows that 
\[\esp{ \frac{1}{S^{2n+1}_{i}} \mathbbm{1}_{S^{2n+1}_{i}\leq \delta \sqrt{n}} } \leq \sum_{k=1}^{\delta \sqrt{n}} C \frac{k}{i ^{3/2}} \leq C \frac{\delta^{2} n}{i ^{3/2}}.
\]
Thus
\[
\proba{ \sum_{i=1}^n  \frac{1}{S^{2n+1}_{i}} \mathbbm{1}_{S^{2n+1}_{i}\leq \delta \sqrt{n}} \geq \e \sqrt{n}} \leq \frac{1}{\e \sqrt{n}} \sum_{i=1}^{n} \esp{ \frac{1}{S^{2n+1}_{i}} \mathbbm{1}_{S^{2n+1}_{i}\leq \delta \sqrt{n}} } \leq  \frac{C}{\e} {\delta^{2}}{\sqrt{n}} \sum_{i=1}^{n} \frac{1}{i ^{3/2}}  \leq  \frac{C}{\varepsilon} \delta^{2}
\]
and \eqref{eq:borne} follows.

\emph{Step 2.} We show that
\begin{equation}
\label{eq:cvloi} X_{n}(\delta) \cvloi X(0,\delta)
\end{equation}
jointy with \eqref{eq:limexc}. Observe that by continuity, for $\eta \in (0,\delta)$,
\[
X_{n}(\delta)-X_{n}(\eta)= \frac{1}{\sqrt{n}} \sum_{i=1}^n  \frac{1}{S^{2n+1}_{i}} \mathbbm{1}_{ \eta \sqrt{n} \leq S^{2n+1}_{i}\leq \delta \sqrt{n}} \cvloi  X(\eta,\delta)
\]
jointly with \eqref{eq:limexc}. Then we have:
 \begin{enumerate}[noitemsep,nolistsep]
\item[--] $X_{n}(\eta) \rightarrow 0$ in probability uniformly in $n$ as $\eta \rightarrow 0$ by \eqref{eq:borne};
\item[--] $X(\eta,\delta) \rightarrow X(0,\delta)$ as $\eta \rightarrow 0$;
\item[--] for every $\eta \in (0,\delta) $ we have $X_{n}(\delta)-X_{n}(\eta) \rightarrow X(\eta,\delta)$ in distribution as $n \rightarrow \infty$. 
\end{enumerate}
This entails \eqref{eq:cvloi}.

\emph{Step 3.} By Skorokhod's representation theorem, we may assume that the convergences \eqref{eq:limexc} and \eqref{eq:cvloi} both hold almost surely. Choose $\delta>0$ such that  $X(0,\delta) \leq \varepsilon$. Then for every $n$ sufficiently large we have $X_{n}(\delta) \leq 2 \varepsilon$. This completes the proof.
\end{proof}

\begin{remark}
By combining Theorem \ref{thm:heightalpha} and Corollary \ref{cor:CRT} we get the following well-known distributional identity (see e.g.~\cite{CMY00}):
\[
\frac{1}{2} \int_{0}^{1} \frac{\mathrm{d} t}{\mathbbm{e}(t)} \quad \mathop{=}^{(d)} \quad \sup \mathbbm{e}.
\]
\end{remark}

\bibliographystyle{abbrv}

\bibliography{bibli}

\appendix

\section{Background on the GP, GH and GHP topologies}
\label{sec:background}
We give background on various topologies that we use, based on \cite{Bla21,Bla22}.
\subsection{The Gromov--Prokhorov (GP) topology} \label{GPdef}
A measured metric space is a triple $(X,d,\mu)$ such that $(X,d)$ is a Polish space and $\mu$ is a Borel probability measure on $X$. Two such spaces $(X,d,\mu)$, $(X',d',\mu')$ are called GP-isometry-equivalent if and only if there exists an isometry $f:\supp(X)\to \supp(X')$ such that if $f_\star \mu$ is the image of $\mu$ by $f$ then $f_\star \mu=\mu'$. Let $\mathbb{K}_{\GP}$ be the set of GP-equivalent classes of measured metric space. Given a measured metric space $(X,d,\mu)$, we write $[X,d,\mu]$ for the GP-isometry-equivalence class of $(X,d,\mu)$ and frequently use the notation $X$ for either $(X,d,\mu)$ or $[X,d,\mu]$.

We now recall the definition of the Prokhorov distance. Consider a metric space $(X,d)$. For every $A\subset X$ and $\e>0$ let $A^\e\coloneqq \{x\in X, d(x,A)<\e\}$ be the open $\varepsilon$-neighborhood of $A$. Then given two (Borel) probability measures $\mu$, $\nu$ on $X$, the Prokhorov distance between $\mu$ and $\nu$ is defined by 
\[d_P(\mu, \nu)\coloneqq \inf\{\text{ $\e>0$: $\mu(A)\leq \nu (A^\e)+\e$ and $\nu(A)\leq  \mu(A^\e)+\e$, for all Borel set $A\subset X$} \}.\]

The  Gromov--Prokhorov (for short GP) distance is an extension of the Prokhorov's distance: For every $(X,d,\mu),(X',d',\mu')\in \K_{\GP}$ the Gromov--Prokhorov distance between $X$ and $X'$ is defined by
\[ d_{\GP}((X,d,\mu),(X',d',\mu'))\coloneqq\inf_{S,\phi,\phi'} d_P(\phi_\star \mu, \phi'_\star\mu'),\]
where the infimum is taken over all metric spaces $S$ and isometric embeddings $\phi :X\to S$, $\phi' :X'\to S$. $d_{\GP}$ is indeed a distance on $\K_{\GP}$ and $(\K_{\GP},d_{\GP})$ is a Polish space (see e.g. \cite{ADH13}).

We use another convenient characterization of the GP topology using the convergence of distance matrices: For every measured metric space  $(X,d^X,\mu^X)$ let $(x_i^X)_{i\in \N}$ be a sequence of i.i.d. random variables of common distribution $\mu^X$ and let $M^X\coloneqq(d^X(x_i^X,x_j^X))_{i,j\in \N}$. We have the following result from \cite{Loh13},

\begin{lemma} \label{equivGP} Let $(X^n)_{n\in \N} \in \K_{\GP}^\N$ and let $X\in \K_{\GP}$ then $X^n\limit^{\GP}X$ as $n\to \infty$ if and only if $M^{X^n}$ converges in distribution toward $M^X$.
\end{lemma}
\subsection{The Gromov--Hausdorff (GH) topology} \label{GH}
Let $\K_{\GH}$ be the set of isometry-equivalent classes of compact metric space. For every metric space $(X,d)$, we write $[X,d]$ for the isometry-equivalent class of $(X,d)$, and frequently use the notation $X$ for either $(X,d)$ or $[X,d]$. 

 For every metric space $(X,d)$, the Hausdorff distance between $A,B\subset X$ is given by
\[d_H(A,B)\coloneqq \inf\{\e>0, A\subset B^\e, B\subset A^\e \}. \]
The Gromov--Hausdorff distance between $(X,d)$,$(X',d')\in \K_{\GH}$ is given by 
\[ d_{\GH}((X,d),(X',d'))\coloneqq\inf_{S,\phi,\phi'} \left (d_H(\phi(X), \phi'(X')) \right ),\]
where the infimum is taken over all metric spaces $S$ and isometric embeddings $\phi :X\to S$, $\phi' :X'\to S$. $d_{\GH}$ is indeed a distance on $\K_{\GH}$ and $(\K_{\GH},d_{\GH})$ is a Polish space (see e.g.  \cite{ADH13}).

\subsection{The Gromov--Hausdorff--Prokhorov (GHP) topology}
Two measured metric spaces $(X,d,\mu)$, $(X',d',\mu')$ are called GHP-isometry-equivalent if and only if there exists an isometry $f:X\to X'$ such that if $f_\star \mu$ is the image of $\mu$ by $f$ then $f_\star \mu=\mu'$.
Let $\K_{\GHP}\subset\K_{\GP}$ be the set of isometry-equivalent classes of compact measured metric space.

 The Gromov--Hausdorff--Prokhorov distance between $(X,d,\mu)$,$(X',d',\mu')\in \K_{\GHP}$ is given by 
\[ d_{\GHP}((X,d,\mu),(X',d',\mu'))\coloneqq\inf_{S,\phi,\phi'} \left ( d_P(\phi_\star \mu, \phi'_\star\mu')+d_H(\phi(X), \phi'(X')) \right ),\]
where the infimum is taken over all metric spaces $S$ and isometric embeddings $\phi :X\to S$, $\phi' :X'\to S$. $d_{\GHP}$ is indeed a distance on $\K_{\GHP}$ and $(\K_{\GHP},d_{\GHP})$ is a Polish space (see  \cite{ADH13}).

Note that random variables which are $\GHP$ measurable are also $\GH$ measurable. For every $[X,d,p]\in \K_{\GHP}$, let $[X,d]_{\GH}$ denote its natural projection on $\K_{\GH}$. Note that GHP convergence implies GH convergence of the projections on $\K_{\GH}$, then that the projection on $\K_{\GH}$ is a measurable function. The same statements hold for the GP topology. We will need the following statement.

\begin{lemma} \label{GP+GH=GHP}[Lemma 4 in \cite{Bla21}] Let $([X^n,d^n,p^n])_{n\in \N}$ and $[X,d,p]$ be GHP measurable random variables in $\K_{\GHP}$. Assume that almost surely $[X,d,p]$ has full support. Assume that $([X^n,d^n,p^n])_{n\in \N}$ converges weakly toward $[X,d,p]$ in a GP sense, and that $([X^n,d^n])_{n\in \N}$ converges weakly toward $[X,d]$ in a GH sense. Then $([X^n,d^n,p^n])_{n\in \N}$ converges weakly toward $[X,d,p]$ in a $\GHP$ sense.
\end{lemma}

\section{The leaf-tightness criterion}
\label{sec:leaf}
In this section, $\mathbf X=((X^n,d^n,p^n))_{n\in \N}$ denotes  a sequence of random compact measured metric spaces GHP-measurable. For every $n\in \N$ let $(x_i^n)_{i\in \N}$ be a sequence of i.i.d. random variables of common distribution $p^n$, then let $M^n\coloneqq(d^n(x_i^n,x_j^n))_{i,j\in \N}$. We say that $\mathbf X$ is leaf-tight if and only if
\begin{equation} \lim_{\delta\to 0} \lim_{k\to \infty} \limsup_{n\to \infty} \proba{d_H(X^n,\{x^n_1,x^n_2,\ldots,x^n_k\})>\delta } = 0. \label{eq:leaf-tight} \end{equation}
Although the criterion is unrelated with leaves, its name ``leaf-tight'' stayed as the original name from Aldous \cite{Ald93} which first introduced this criterion to study random trees with random leaves. 

\begin{proposition}
\label{C.1}
If $(M^n)_{n\in \N}$ converges weakly toward a random matrix $M$, and $\mathbf X$ is leaf-tight, then $\mathbf X$ converges weakly for the GHP topology toward a random compact measured metric space $(X,d,p)$. Furthermore, if $(x_i)_{i\in \N}$ are i.i.d. random variables of law $p$ then $M^X\coloneqq(d(x_i,x_j))_{i,j\in \N}=^{(d)} M$. In addition, a.s. $p$  has full support.
 \end{proposition}
 
\begin{proof}
We focus on the first statement, as the other two will naturally follow alongside the proof. First, it directly follows from the convergence of $(M_n)_{n\in \N}$ and from \eqref{eq:leaf-tight} that: 
\begin{itemize}
\item $(\diam(X^n))_{n\in \N}$ is tight. 
\item For every $\delta>0$, $(N_\delta(X^n))_{n\in \N}$ is tight, where $N_\delta$ stands for the usual  $\delta$-covering number (see \cite{ADH13})
\end{itemize}
As a result, by \cite[Theorem 2.4]{ADH13}, $\mathbf X$ is tight for the weak GHP-topology.

Let $[X,d,p]$, $[X',d',p']$ be two GHP subsequential limits of $\mathbf X$. Let us show $[X,d,p]=^{(d)}[X',d',p']$. 
  Let $(x_i)_{i\in \N}$ be i.i.d. random variables of law $p$ and let $M^X\coloneqq(d(x_i,x_j))_{i,j\in \N}$. Define similarly $(x'_i)_{i\in \N}$,  $M^{X'}$. By Lemma
\ref{equivGP}, jointly with the above convergence,  along the proper subsequence, we have weakly $M^n\limit M^X$ and $M^n\limit M^{X'}$. Note that this implies $M^X=^{(d)} M$. By the Skorokhod representation theorem, and up to further extract other subsequences, we may assume that those two convergences hold almost surely, and that almost surely $M^X=M^{X'}$. This implies that almost surely,
\[ d_{GH}((\{x_i\}_{i\in \N},d),(\{x'_i\}_{i\in \N},d)) =0. \]
And by the leaf-tightness criterion \eqref{eq:leaf-tight}, almost surely $\{x_i\}_{i\in \N}$ and $\{x'_i\}_{i\in \N}$ are dense on $X$, $X'$ so 
\[ d_{GH}((X,d),(X',d')) =0. \]
Hence $[X,d]_{GH}=^{(d)}[X',d']_{GH}$ . Also by Lemma \ref{equivGP}, since $(M^n)_{n\in \N}$ converges weakly, we have the equality in distribution $[X,d,p]_{GP}=^{(d)}[X',d',p']_{GP}$. Moreover, since almost surely $\{x_i\}_{i\in \N}$ and $\{x'_i\}_{i\in \N}$ are dense on $X$ and $X'$, almost surely $p$ and $p'$ have full support on $X$ and $X'$. Therefore, by Lemma \ref{GP+GH=GHP} we have $[X,d,p]=^{(d)}[X',d',p']$. This shows that there is only one subsequential possible limit for $\mathbf X$, and thus concludes the proof.
\end{proof}

\end{document}